\documentclass[11pt]{amsart}

\usepackage[utf8]{inputenc}

\usepackage{stmaryrd}                  
\usepackage{footmisc}                  
\usepackage{paralist}
\usepackage[draft]{pdfcomment}
\usepackage{bbm}
\usepackage[british]{babel}
\usepackage{amsthm}
\usepackage{color} 
\usepackage{enumerate} 
\usepackage{graphicx} 
\usepackage{subcaption} 
\usepackage{tikz-cd} 
\usepackage{verbatim} 
\usepackage{hyperref}
\usepackage[percent]{overpic} 
\usepackage{bm} 
\usepackage{float} 
\usepackage{titleps}

\newtheorem*{theorem*}{Theorem}
\newtheorem{theorem}{Theorem}
\newtheorem{lemma}[theorem]{Lemma}
\newtheorem{prop}[theorem]{Proposition}
\newtheorem{cor}[theorem]{Corollary}

\newtheorem*{thmA}{Theorem A}
\newtheorem*{thmB}{Theorem B}
\newtheorem*{thmC}{Theorem C}

\newtheoremstyle{myremark}{}{9pt}{\upshape}{}{\bfseries}{.}{ }{}
\theoremstyle{myremark}
\newtheorem{remark}[theorem]{Remark}

\theoremstyle{definition}
\newtheorem{defn}[theorem]{Definition}

\newcommand{\R}{\mathbb{R}}

\newcommand{\N}{\mathbb{N}}
\newcommand{\Z}{\mathbb{Z}}

\newcommand{\diff}{d}

\newcommand{\vol}{\mathrm{vol}\,}

\newcommand{\Oo}{\mathcal{O}}

\newcommand{\Or}{\mathrm{O}}

\newcommand{\To}{\rightarrow}

\newcommand{\Orb}{\mathcal{O}}
	
\definecolor{figureblue}{RGB}{0,0,100} 
\definecolor{figurered}{RGB}{200,0,0}

\title[Sharp systolic inequalities for rotationally symmetric $2$-orbifolds]{Sharp systolic inequalities for \\ rotationally symmetric $2$-orbifolds}

\author[C. Lange]{Christian Lange}
\address{Christian Lange\newline\indent Mathematical Institute of the University of Munich\newline\indent Theresienstr. 39, D-80333 Munich, Germany}
\email{lange@math.lmu.de, clange.math@gmail.com}

\author[T. Soethe]{Tobias Soethe}
\address{Tobias Soethe\newline\indent Ruhr-Universit\"at Bochum, Fakult\"at f\"ur Mathematik\newline\indent Universitätsstr. 150, D-44780 Bochum, Germany}
\email{tobias.soethe@ruhr-uni-bochum.de}


\begin{document}

\begin{abstract} We show that suitably defined systolic ratios are globally bounded from above on the space of rotationally symmetric spindle orbifolds and that the upper bound is attained precisely at so-called Besse metrics, i.e. Riemannian orbifold metrics all of whose geodesics are closed.
\end{abstract}\maketitle	

\section{Introduction}\label{sec:Introduction}

Systolic geometry studies the relation between the length of shortest closed geodesics and the volume of the ambient space. The systole of a Riemannian $2$-sphere is defined as the length of its shortest nontrivial closed geodesic. Its systolic ratio is the square of the systole divided by the sphere's area. This ratio can become arbitrarily small. On the other hand, it is globally bounded from above due to a result of Croke \cite{croke88}. The optimal upper bound is conjectured to be $2\sqrt 3$, which is approached for a sequence of spheres converging to the so-called Calabi-Croke singular sphere, two copies of  an equilateral triangle glued together along their boundary \cite{croke88}.

A Riemannian metric on $S^2$ is called Zoll if all its prime geodesics are closed and have the same length. The round metric on $S^2$ is one example of such a Zoll metric, but the space of Zoll metrics on $S^2$ is in fact infinite-dimensional. Already in 1908 an infinite family of rotationally symmetric Zoll metrics on $S^2$ was constructed by Zoll \cite{Zoll:1908}. The systolic ratio of any Zoll metric on $S^2$ is $\pi$, see e.g. \cite{abbondandolo17}. In \cite{abbondandolo17} Abbondandolo et al. show that for spheres with positive- and sufficiently pinched curvature the systolic ratio is bounded from above by $\pi$, and that the upper bound is attained if and only if the sphere is Zoll. For general Riemannian spheres they obtain this conclusion in a $C^3$-neighborhood of any Zoll metric on $S^2$ \cite{abbondandolo18_2}. Later, they proved that on spheres of revolution the systolic ratio is bounded from above by $\pi$ with equality if and only if the metric is Zoll \cite{abbondandolo18}.

A notion less restrictive than Zoll is that of Besse metric or flow. A (geodesic) flow (and its defining metric) is called Besse, if all its orbits are periodic. This definition does not impose constraints on the periods. However, it implies the existence of a common period under very general assumptions, for instance for Reeb flows \cite{Sullivan:1978}. A conjecture of Berger states that every Besse Riemannian metric on a simply connected manifold is Zoll. This conjecture was confirmed for $S^2$ by Gromoll and Grove \cite{gromoll81} and for $S^n$ with $n\ge 4$ by Radeschi and Wilking \cite{radeschi17}. On the other hand, there are simply connected orbifolds that admit Besse metrics which are not Zoll. For instance, such Besse orbifolds can be constructed as surfaces of revolution homeomorphic to $S^2$ with two cyclic orbifold singularities \cite{besse78}. We call such orbifolds \emph{spindles} and denote them as $S^2(m,n)$, where $m$ and $n$ are the order of the two cyclic singularities. In this case one can show that the shortest closed geodesic is an equator and that any other prime closed geodesic is $\frac{m+n}{2-\alpha}$-times longer than this equator, with $\alpha=0$ for $m+n$ even and $\alpha=1$ for $m+n$ odd. This is a special case of a more general rigidity result about the length spectrum of a Besse $2$-orbifold \cite{lange20}.

One might expect that Besse metrics can be characterized as global maximizers of the systolic ratio among rotationally symmetric $2$-orbifolds. However, this naive generalization fails, see Section 3.\ref{sec:surfacerevolution}. Nevertheless, we are able to obtain a generalization as follows. The unit tangent bundle $T^1\Oo$ of a spindle orbifold $\Oo= S^2(m,n)$ is a smooth $3$-manifold diffeomorphic to a lens space of type $L(m+n,1)$ \cite{lange20}, and as such it has fundamental group $\Z_{m+n}$. We define the contractible systolic ratio as the quantity
 \begin{align*}
  \rho_\text{contr}(\Oo) = \frac{\ell_\text{min,contr}^2}{\mathrm{area}(\Oo)} \, ,
 \end{align*}
 where $\ell_\text{min,contr}$ denotes the length of the shortest closed geodesic whose lift to the unit tangent bundle is contractible. For the contractible systolic ratio we prove the following result.

\begin{thmA}\label{thm:A}
 Let $\Oo=S^2(m,n)$ be a rotationally symmetric spindle orbifold. Then the contractible systolic ratio is bounded from above by
 \begin{align*}
  \rho_{\mathrm{contr}}(\Oo) \le 2(m+n)\pi \, .
 \end{align*}
 Moreover, the upper bound is attained if and only if $\Oo$ is Besse.
\end{thmA}

Theorem A constitutes a new result even in the smooth case: On smooth rotationally symmetric spheres the Zoll metrics are global maximizers of the contractible systolic ratio. We point out that, in contrast to the ordinary systolic ratio, no other (not necessarily rotationally symmetric) metrics with higher contractible systolic ratio on $S^2$ are known. For the contractible systolic ratio on $S^2$ Zoll metrics are conjectured to be global maximizers, cf. the discussion after Corollary 4 in \cite{abbondandolo18_2} and the references therein.

The following generalization of the contractible systolic ratio was suggested to us by P. A. S. Salom\~{a}o: For a divisor $k$ of $(m+n)$ we consider the quantity
 \begin{align*}
  \rho_{\mathrm{contr},k}(\Oo) = \frac{\ell_\text{min,k}^2}{\mathrm{area}(\Oo)} \, ,
 \end{align*}
 where $\ell_\text{min,k}$ denotes the length of the shortest closed geodesic whose lift to the unit tangent bundle $T^1\Oo$ represents an element in the subgroup of $\pi_1(T^1\Oo)$ of order $k$. We remark that $\rho_\text{contr}=\rho_{\text{contr},1}$.
  We can consider the geodesic flow as a Reeb flow on the unit tangent bundle. Then for a suitable covering of $L(m+n,1)$, the systolic ratio $\rho_{\text{contr},k}$ coincides with the standard systolic ratio of the lifted Reeb flow up to a multiplicative constant. For Besse spindle orbifolds the only such lifts that are Zoll are the lifts to the universal covering $S^3$ and, if $n+m$ is even, to the $\frac{m+n}{2}$-fold covering $L(2,1)$.
	
We obtain the following systolic inequality, which for $m=n=1$ recovers the result about the standard systolic ratio on spheres of revolution in \cite{abbondandolo18}.

 \begin{thmB} \label{thm:B}
 Let $\Oo=S^2(m,n)$ be a rotationally symmetric spindle orbifold with $m+n$ even. Then we have
 \begin{align*}
  \rho_{\mathrm{contr},2}(\Oo) \le \frac{m+n}{2}\pi \, 
 \end{align*}
 and the upper bound is attained if and only if $\Oo$ is Besse.
\end{thmB}

For all values of $k$ not covered by Theorem A and B, Besse metrics even fail to be local maximizers of the systolic ratio $\rho_{\text{contr},k}$, see Remark \ref{rmk:no_local_maximizer_of_rho_sysk}.\\

The above mentioned fact that Zoll metrics on $S^2$ are local maximizers in the $C^3$-topology is a corollary of a corresponding statement about Zoll Reeb flows. If $\tau_1(\lambda)$ denotes the minimum of all periods of closed Reeb orbits of a closed, connected contact $3$-manifold $(Y,\lambda)$, then Zoll Reeb flows are local maximizers in the $C^3$-topology of the systolic ratio \cite{abbondandolo18_2,Benedetti:2021aa,Abbondandolo:2019tl}
\begin{align}
\label{e:1st_systolic_ratio} 
\rho_1(\lambda) := \frac{\tau_1(\lambda)^2}{\vol(Y,\lambda)}.
\end{align}
Here $\vol(Y,\lambda)$ is the contact volume defined as the integral of the volume form $\lambda\wedge \diff\lambda$ over $Y$, which in case of the unit tangent bundle of a surface equals $2\pi$ times the surface's area \cite[Proposition~3.7]{abbondandolo17}. Note that on a closed contact $3$-manifold a closed geodesic always exists by Taube's proof of the Weinstein conjecture in dimension $3$ \cite{Taubes:2007wi}. Similarly, Abbondandolo, Mazzucchelli and the first named author have recently characterized Besse Reeb flows on closed, contact $3$-manifolds as local maximizers of the higher systolic ratios
\[
\rho_k(\lambda) := \frac{\tau_k(\lambda)^2}{\mathrm{vol}(Y,\lambda)},
\]
where $\tau_k(\lambda)$ is, roughly speaking, the $k$-th shortest period in the period spectrum of the Reeb flow \cite{Abbondandolo:2021tl}.

As our third result we obtain a corresponding global sharp upper bound in the class of rotationally symmetric spindle orbifolds. To make this precise, we denote by $\sigma(\Orb)$ the period spectrum of the geodesic flow $\phi^t$ of a rotationally symmetric spindle orbifold $\Orb$, i.e.\ the set
\[
\sigma(\Orb)=\big\{t>0\ \big|\ \mathrm{fix}(\phi^t)\neq \emptyset \big\},
\]
and by $\tau_k(\Orb)$ the infimum of all positive real numbers $\tau$ such that there exist at least $k$ closed geodesics with period less than or equal to $\tau$. In formulas,
\begin{equation}
\label{e:tau_k}
\tau_k(\Orb):=\inf\Bigg\{\tau>0\ \Bigg|\ \sum_{0<t\leq\tau} \#\big(\mathrm{fix}(\phi_\lambda^t)/\sim\big)\geq k  \Bigg\},
\end{equation}
where $\sim$ is the equivalence relation on the unit tangent bundle of $\Orb$ which identifies points on the same orbit of the geodesic flow. Note that the sequence of values $\tau_k(\lambda)$, $k\geq1$, is (not necessarily strictly) increasing and consists of elements of $\sigma(\lambda)$.  
Finally, the $k$-th systolic ratio of $\Orb$ is defined as the positive number
\[
\rho_k(\Orb) := \frac{\tau_k(\Orb)^2}{\mathrm{area}(\Orb)}.
\]
Note that this definition differs from the one given in \cite{Abbondandolo:2021tl} by a factor of $2\pi$. We prove the following result.

\begin{thmC}\label{thm:C}
 Let $\Oo=S^2(m,n)$ be a rotationally symmetric spindle orbifold. Then we have
 \begin{align*}
  \rho_{\frac{n+m}{2-\alpha}}(\Oo) \le \frac{2(m+n)\pi}{(2-\alpha)^2} \, .
 \end{align*}
 Moreover, the upper bound is attained if and only if $\Oo$ is Besse.
\end{thmC}

We point out that the systolic ratio considered in Theorem C is in general distinct from the ones considered in Theorem A and B. For instance, if a rotationally symmetric metric on $S^2(m,n)$ has an alternating and sufficiently long sequence of very long equators and very short equators of the same size, then the systolic ratio considered in Theorem C will be smaller than the ones considered in Theorem A and B.

The analytic methods used in the proofs of Theorem A, B and C are the same as in \cite{abbondandolo18}. The main additional difficulty compared to the result in \cite{abbondandolo18} is the identification and use of the correct systolic ratios which requires e.g. topological arguments.

In Section \ref{chapter:prelim2} we recall basic facts about Riemannian orbifolds, the dynamics on rotationally symmetric surfaces and the topology of the unit tangent bundle in the orbifold case. In Section \ref{sec:surfacerevolution} we introduce the contractible systolic ratio and establish some useful properties that are needed in this context. In Section \ref{sec:generatingfunction} and \ref{sec:generatingfunction_prop} we recall and generalize facts about generating functions. Finally, in Section \ref{sec:systolicineq} we prove our main results. \\

\textit{Acknowledgements:} The second named author is very grateful to his advisor \textit{Alberto Abbondandolo} both for pointing out the topic of interest treated in this paper and for his encouragement and support. He is partially supported by the DFG funded project SFB/TRR 191 ‘Symplectic Structures in Geometry, Algebra and Dynamics’ (Projektnummer 281071066 – TRR 191). The first named author thanks Alberto Abbondandolo for suggesting this collaboration. He was partially supported by the SFB/TRR 191 as well.

\section{Preliminaries}\label{chapter:prelim2}

\subsection{Riemannian orbifolds}\label{sec:orbifolds}
A \emph{length space} is a metric space in which the distance of any two points can be realized as the infimum of the lengths of all rectifiable paths connecting these points. An \emph{$n$-dimensional Riemannian orbifold} is a length space $\Orb$ such that each point in $\Orb$ has a neighborhood that is isometric to the quotient of an $n$-dimensional Riemannian manifold by an isometric action of a finite group. Every such Riemannian orbifold has a canonical smooth orbifold structure in the classical sense \cite{lange20_orb}. Conversely, every smooth orbifold can be endowed with a Riemannian metric, and then the induced length metric turns it into a Riemannian orbifold in the above sense.

For a point $x$ on a Riemannian orbifold, the isotropy group of a preimage of $x$ in a Riemannian manifold chart is uniquely determined up to conjugation. Its conjugacy class in $\Or(n)$ is called the \emph{local group} of $\Orb$ at $x$. The point $x$ is called \emph{regular} if this group is trivial and \emph{singular} otherwise. We denote the union of all singular point in $\Orb$ as $\Sigma(\Orb)$. In this paper we only work with $2$-dimensional orbifolds. In this case only cyclic groups generated by rotations and dihedral groups generated by reflections can occur as local groups. Hence, the underlying topological space is a manifold with boundary in this case, and the boundary consists precisely of those points whose local groups contain a reflection. A $2$-orbifold is \emph{orientable} if and only if its underlying surface has no boundary and is orientable. More specifically, we work with Riemannian $2$-orbifolds whose underlying topological space is a sphere, and so in this case only local groups generated by rotations occur. Such a local group is uniquely determined by its order. In this case the underlying smooth orbifold is uniquely determined by the orders $n_1,\ldots,n_k$ of the local groups, and we denote it as $S^2(n_1,\ldots,n_k)$.

We would like to compare the volume, i.e. the area in the present case, of such a Riemannian orbifold with squared lengths of closed geodesics on it. The \emph{volume} of an $n$-dimensional Riemannian orbifold $\Orb$ can for instance be defined as the volume of its regular part, which is a Riemannian manifold, i.e. $\vol(\Oo) := \vol_g(\Oo\setminus\Sigma(\Oo))$. Alternatively, we could define it as the $n$-dimensional Hausdorff measure. 

An \emph{(orbifold) geodesic} on a Riemannian orbifold is a continuous path that can locally be lifted to a geodesic in a Riemannian manifold chart. A \emph{closed geodesic} is a continuous loop that is a geodesic on each subinterval. A \emph{prime geodesic} is a closed geodesic that is not a concatenation of nontrivial closed geodesics. By the \emph{length} of a closed geodesic we mean its length as a parametrized curve.
We point out that this notion may differ from the length of the geometric image of the curve. To exemplify this and to provide some intuition for orbifold geodesics let us record some of their properties. 

Away from the singular part geodesics behave like  geodesics in Riemannian manifolds. A geodesic that hits an isolated singular point either passes straight through it or is reflected at it depending on whether the order of the corresponding local group is odd or even. We see that a closed geodesic that hits an isolated singular point of even order traverses its trajectory twice during a single period. This also shows that an orbifold geodesic is in general not locally length minimizing. On the other hand, a locally length minimizing path is always a geodesic in the orbifold sense.

In particular, we will be concerned with Riemannian $2$-orbifolds homeomorphic to $S^2$ that admit an effective $S^1$-action by isometries. In this case there can be at most two singular points, i.e. the orbifold is of type $S^2(m,n)$, and the singular points are fixed by the $S^1$-action. We refer to such an orbifold as a \textit{spindle orbifold}. Spindle orbifolds with $m\neq n$ play a special role among $2$-orbifolds in that they are the only orientable $2$-orbifolds that are not developable, i.e. they can not be realized as a global quotient of a manifold by a finite group action.

\subsection{Rotationally symmetric metrics} \label{sub:rotationally_symmetric_metrics}
Suppose $\Orb$ is a rotationally symmetric Riemannian $2$-orbifold homeomorphic to $S^2$. In other words, we have an effective and isometric $S^1$-action on $\Orb$. It follows for instance by the work of Mostert \cite{Moster:1957} that this action is conjugated to a linear action. In particular, it has precisely two fixed points. Since all singular points on $\Orb$ are isolated, we see that at most the two fixed points can be singular. By compactness there exists a minimizing geodesic $\sigma$ that connects the two fixed points. Suppose that $\sigma$ is parametrized by arclength on $[0,M]$, i.e. $\sigma(0)$ and $\sigma(M)$ are the two fixed points of the $S^1$-action. Then the map defined by
\begin{align*}
 \Phi: \R/2\pi\Z \times (0,M) \to \Orb \setminus \{\sigma(0),\sigma(M)\} \, , \ \ \ \Phi(\theta,s) = e^{i\theta}\sigma(s) \, .
\end{align*}
provides us with a smooth parametrization of the regular part of $\Orb$. With respect to this parametrization the metric attains the form
\begin{align}\label{eq:metricrotation}
 g = r(s)^2 d\theta^2 + ds^2 \, 
\end{align}
for some smooth function $r:(0,M) \To (0,\infty)$ such that $r(s)$ converges to $0$ when $s$ tends to $0$ or $M$. Moreover, the smoothness assumption implies additional boundary conditions, cf. Proposition \ref{thm:tannery_orbifold}.

Metrics of the form (\ref{eq:metricrotation}) for instance occur as \emph{(singular) spheres of revolution} in $\R^3$. Such a sphere of revolution $S$ that is invariant under rotations around the $z$-axis is uniquely determined by its intersection with the half-plane $\{(x,0,z)\in\R^3 \,|\, x\ge 0 \}$. This intersection has to be a smooth embedded curve which we will call the \textit{enveloping curve} and denote by $\sigma$. The fact that $S$ is homeomorphic to the sphere implies that there are only two points at which $\sigma$ makes contact with the axis of rotation (namely at its starting and end point). We denote the length of $\sigma$ by $M$ and give it an arclength parametrization $\sigma:[0,M] \to \R^2$. We denote the components of the curve by
\begin{align*}
 \sigma(s) = (r(s),z(s)) \, , \quad s \in [0,M] \, .
\end{align*}
Then $r$ is a non-negative function and it equals zero only for $s=0$ and $s=M$. The surface $S$ can be recovered from $\sigma$ by rotating the curve around the $z$-axis which produces the set
\begin{align*}
 S := \{(r(s)\cos\theta, r(s)\sin\theta, z(s)) \in \R^3 \,|\, \theta \in \R/2\pi\Z, s \in [0,M] \} \, .
\end{align*}
The induced Riemannian metric on the smooth part $U:=S \backslash \{\sigma(0),\sigma(M)\}$ of $S$ takes the form (\ref{eq:metricrotation}) with respect to the coordinates $s$ and $\theta$. Its completion is a Riemannian orbifold metric if and only if the function $r$ satisfies certain boundary conditions, cf. Proposition \ref{thm:tannery_orbifold}. In particular, the surface $S$ is smooth if and only if $\sigma$ extends to a smooth curve by mirroring it along the vertical axis in $\R^2$.

For now we keep the discussion general and work with metrics of the form (\ref{eq:metricrotation}) on an open annulus $U=\R / 2\pi \times (0,M)$ without regard to further smoothness assuring boundary conditions. We call those curves \textit{parallels} which are arclength parametrizations of the circles
\begin{align*}
 P_s := \{(r(s)\cos\theta,r(s)\sin\theta,z(s)) \,|\, \theta \in \R/2\pi\Z \} \, , \quad s \in (0,M)
\end{align*}
and orient them counterclockwise. If $s$ is a critical point of $r$, we call the corresponding parallel $P_s$ an \textit{equator}. Parallels are closed geodesics if and only if they are equators. We call curves \textit{meridional arcs} if they are unit speed parametrizations on $(0,M)$ with constant $s$-component. Meridional arcs are always geodesic arcs. In the case of $S$ being a smooth surface two such meridional arcs can be concatenated to form a closed geodesics which we call a \textit{meridian}. In the orbifold case the behaviour of geodesics through the singular points can be different as discussed in Section \ref{sec:orbifolds} and further below in this section.

We consider the tangent bundle $T^1U$ over $U$. On $U$ the metric is well defined and we denote the geodesic flow by $\phi$. Because of the singular points, the flow is not globally defined but rather as a map
 \begin{align*}
  \phi:\Omega \to T^1U \, , \quad (t,x) \mapsto \phi_t(x)
 \end{align*}
on a suitable maximal open neighbourhood $\Omega$ of $\{0\} \times T^1U$ in $\R\times T^1U$. By the \textit{Hilbert contact form}, which we denote by $\alpha$, we refer to the contact form on $T^1U$ which is obtained by restricting the canonical Liouville form of the cotangent bundle of $U$ to the unit cotangent bundle and then pulling it back to $T^1U$ by the bundle isomorphism which is induced by the metric. 

We parametrize the unit tangent bundle over $U$ as follows: Let $u$ be a unit tangent vector to $U$ at a point $p = \varphi(\theta,s)\in U$. We denote by $\beta(u)$ the angle formed by $u$ and the positive direction of the parallel $P_s$ passing through $p$. By taking \eqref{eq:metricrotation} into account, we get that $T^1U$ is the image of the diffeomorphism
\begin{align}
 &\Psi: \R/2\pi\Z \times \R/2\pi\Z \times (0,M) \to T^1U \, , \nonumber \\
 &\Psi(\theta,\beta,s) = \left( \Phi(\theta,s), \frac 1r \cos\beta \frac{\partial\Phi}{\partial\theta}(\theta,s) + \sin\beta \frac{\partial\Phi}{\partial s}(\theta,s) \right) \, . \label{eq:coord_unittangent}
\end{align} 
In the coordinate system \eqref{eq:coord_unittangent}, the Hilbert contact form can be calculated to be
\begin{align*}
 \alpha(\theta,\beta,s) = r(s) \cos\beta d\theta + \sin\beta ds \, .
\end{align*}
It follows that the contact volume form on $T^1U$ in these coordinates is given by
\begin{align}\label{eq:hilbertcontactform}
 (\alpha \wedge d\alpha)(\theta,\beta,s) = r(s)\, d\theta \wedge d\beta \wedge ds \, .
\end{align}
The Reeb vector field of the Hilbert contact form $\alpha$ takes the form
\begin{align*}
 R(\theta,\beta,s) = \frac{\cos\beta}{r(s)}\frac{\partial}{\partial\theta} + \frac{r'(s)\cos\beta}{r(s)}\frac{\partial}{\partial\beta} + \sin\beta \frac{\partial}{\partial s} \, .
\end{align*}
The geodesic flow $\phi:\Omega \to T^1U$ coincides with the Reeb flow of the Hilbert contact form. Consequently, by the above expression for the Reeb vector field, the geodesic equation is given by the following system of equations:
\begin{align}\label{eq:geodesiceqations}
\begin{split}
 \dot\theta &= \frac{\cos\beta}{r(s)} \, , \\[1ex]
 \dot\beta  &= \frac{r'(s)\cos\beta}{r(s)}  \, ,\\[0.5ex]
 \dot s		&= \sin\beta \, .
\end{split}
\end{align}
From the geodesic equation, one recovers a first integral of the geodesic flow which is known as the \textit{Clairaut function}. We make this statement precise in the form of the following lemma, whose proof is a direct calculation using the geodesic equations in the form of (\ref{eq:geodesiceqations}).
\begin{lemma}
 Let $\phi:\Omega \to T^1U$ be the geodesic flow of the metric (\ref{eq:metricrotation}) on $U$. Denote by
 \begin{align*}
 K: T^1U \to \R \, , \quad K(u) = K(\theta,\beta,s) := r(s)\cos\beta
 \end{align*}
 the so-called Clairaut function. Then for all $u \in T^1U$ we have
 \begin{align*}
  \frac{\partial}{\partial t} K(\phi_t(u)) = 0 \, .
 \end{align*}
\end{lemma}

We note that if for a geodesic $\gamma: (a,b)\to S$ we have that $K(\dot\gamma)=0$, the geodesic has to be a meridonal arc and will only be defined on a finite time interval. All other geodesics in the regular part of $S$ corresponding to positive or negative values of the Clairaut function, will be defined for all times. We record some facts about the behaviour of geodesics that are not meridional arcs in the following two lemmas. The first lemma tells us that such a geodesic winds around the axis of revolution in a fixed rotational direction.

\begin{lemma}\label{lemma:thetaderivative}
 Let $\gamma:\R\to S$ be a geodesic that does not coincide with a meridional arc. Then the time derivative of the angle-$\theta$-coordinate is either positive or negative along the geodesic.
\end{lemma}

\begin{proof}
 We already discussed that for geodesics that are not meridional arcs, the Clairaut function is either positive or negative. This is equivalent to the desired statement by the first geodesic equation in \eqref{eq:geodesiceqations}.
\end{proof}

The second lemma characterizes the geodesics which are not meridional arcs.

\begin{lemma}\label{lemma:geodesics_surfacerev}
 Let $\gamma:\R\to S$ be a geodesic that does not coincide with a meridional arc. Then one of the following two alternatives hold.
 \begin{enumerate}[i)]
 \item (Asymptotic geodesic) for $t \to -\infty$ and $t \to + \infty$ the geodesic $\gamma$ is asymptotic to two possibly coinciding equators $P_{s_-}$ and $P_{s_+}$ with $r(s_-)=r(s_+)=|K(\dot\gamma)|$;
 \item (Oscillating geodesic) the geodesic $\gamma$ oscillates between two parallels $P_{s_1}$ and $P_{s_2}$. More precisely, there exist numbers $0 < s_1 < s_2 < M$ such that
 \begin{align*}
  r(s_1)=r(s_2)=|K(\dot\gamma)|<r(s) \quad \forall s \in (s_1,s_2) \, , \quad r'(s_1)>0 \, , \ r'(s_2) < 0 \, ,
 \end{align*}
 $\gamma$ is confined to the strip
 \begin{align*}
  \bigcup_{s\in[s_1,s_2]} P_s \, ,
 \end{align*}
 and it alternately touches the parallels $P_{s_1}$ and $P_{s_2}$ tangentially infinitely many times.
 \end{enumerate}
\end{lemma}

\begin{proof}
 The proof given in \cite[Proof of Lemma~1.1]{abbondandolo18} for $S$ being smooth everywhere extends to our more general setup, as we only consider geodesics in the regular part.
\end{proof}

We conclude this section by discussing of how to extend the meridional arcs to closed geodesics in the case of rotationally symmetric spindle orbifolds. Like in the smooth case, we again call those geodesics \textit{meridans}. We already remarked that a geodesic that hits an isolated singular point either passes straight through it or is reflected depending on whether the order of the corresponding local group is odd or even. That means, if we distinguish all possible cases, a \textit{meridian} is the curve of the following form: for $m$ and $n$ odd, it is a curve obtained by concatenating a meridional arc $\gamma_m$ with the antipodal meridional arc $\gamma_m^a$ that is parametrized in a direction that makes the resulting curve closed; for $m$ and $n$ even, a meridian is a curve obtained by concatenating a meridional arc $\gamma_m$ with $\gamma_m^{-1}$, that is, the inverse parametrization of itself; for $m+n$ odd a meridian is a curve of the form $\gamma_m \ast \gamma_m^{-1} \ast (\gamma_m^a)^{-1} \ast \gamma_m^a$.

\subsection{Besse orbifolds}\label{sec:besseorbifolds}

A Riemannian orbifold is called \emph{Besse} if all its geodesics are closed. In this case we also call the Riemannian metric itself Besse. A Besse $2$-orbifold is always compact \cite[Propositon~2.6]{lange20}. Examples of Besse orbifolds are quotients of the round $2$-sphere $S^2$ by finite subgroups of $\mathrm{O}(3)$. Such orbifold are called \emph{spherical}. Other examples of Besse metrics exist on spindle orbifolds. Such examples can either be constructed as quotient metrics for weighted Hopf action on $S^3$, see e.g. \cite{Guillemin09}, or as rotationally symmetric metrics along the lines of the previous section. The latter source yields an infinite-dimensional space of examples which we will discuss in more detail below. It is conjectured that also the space of non-rotationally symmetric Besse metrics on spindle orbifolds is infinite dimensional. Conversely, an orientable Besse $2$-orbifold is either spherical or a spindle orbifolds \cite[Propositon~2.7]{lange20}. An interesting property of Besse $2$-orbifolds is that the length spectrum of the closed geodesics is determined up to scaling by the underlying (smooth) orbifold  \cite{lange20}. This observation generalizes a result of Gromoll and Grove according to which a Besse metric on the $2$-sphere is actually Zoll, i.e. all its prime geodesics have the same length (and are simple closed). For results and more information on Besse orbifolds in higher dimensions we refer the reader to \cite{Amann18}.

We now turn to rotationally symmetric metrics on spindle orbifolds as discussed in the previous section. In particular, we will recall a characterization of Besse metrics among them. But first we recall the following fact.

\begin{lemma}\label{lemma:uniqueequator}\cite[Lemma~4.9]{besse78}
 Let $S^2(m,n)$ be a rotationally symmetric Besse spindle orbifold. Then there exists a unique equator (up to reversing the orientation).
\end{lemma}

In light of the preceding lemma, for the rest of this section, whenever we discuss rotationally symmetric Besse spindle orbifold, we normalize the metric such that the unique equator has length $2\pi$. In this case the length of all other prime geodesics is constant and does only depend on the orders $m$ and $n$ of the two singular points. This can be deduced from \cite[Theorem~4.13]{besse78} and has been proven in greater generality without the symmetry assumption in \cite{lange20}. In the following proposition we state this result together with a classification of rotationally symmetric Besse spindle orbifolds that also follows from \cite[Theorem~4.13]{besse78}.

\begin{prop}\label{thm:tannery_orbifold}
 Let $S^2(m,n)$ be a spindle orbifold with a rotationally symmetric Riemannian metric. It has all of its geodesics closed with minimal common period $2\pi$ if and only if on the regular part $U$ there exists a coordinate system $(R(s),\theta)$ such that the metric takes the form
 \begin{align}
  g = \left( \frac{m+n}{2} + h(\cos R) \right)^2 dR^2 + \sin^2 R\, d\theta^2 \, ,
 \end{align}
 where $h$ is an odd smooth function from $(-1,1)$ to $(-\frac{m+n}{2}, \frac{m+n}{2})$ that extends to a continuous function on $[-1,1]$ with $h(-1)=\frac{m-n}{2}=h(1)$.

 In this case the length of the equator is $2\pi$ and the length of any other prime geodesic is $2\frac{m+n}{2-\alpha} \pi$, where
 \begin{align*}
  \alpha = \begin{cases} 0 \, , & m+n \text{ even} \\ 1 \, , & m+n \text{ odd} \end{cases} \, .
 \end{align*}
 Any geodesic in the regular part of $S$ other than the equator oscillates $(1+\alpha)$ times between two parallels and makes $\frac{m+n}{2-\alpha}$ full revolutions around the singular points before closing up.
\end{prop}


A version of this statement for more general rotationally symmetric surfaces is stated in \cite[Theorem~4.13]{besse78}.

\subsection{First return map and first return time}\label{ssec_returnmapreturntime}
Let $\Oo$ be a $2$-orbifold homeomorphic to $S^2$ with a rotatationally symmetric metric as discussed in Section \ref{sub:rotationally_symmetric_metrics}. In order to prove a version of \cite{abbondandolo17} in the orbifold case, we will need to look at the return map of a suitable equator. More precisely, we consider an equator corresponding to a critical point $s_0 \in (0,M)$ of $r$ such that $r(s_0)$ is minimal among all equators. We define the \textit{open Birkhoff annulus} associated to the aforementioned equator (given a positive orientation) to be the set
\begin{align*}
 A := \{u \in T^1\Oo \,|\, \beta(u) \in (0,\pi), p(u) \in P_{s_0} \} \, ,
\end{align*}
where $p: TM \to M$ denotes the projection map. The geodesic flow is transverse to $A$, as the equator itself is a closed geodesic. 

\begin{lemma}\label{lemma:birkhoffannulusevolutions}\cite[Lemma~2.1]{abbondandolo18}
 Let $s_0 \in (0,M)$ be a critical point of $r$ such that $r(s_0)$ is minimal among the critical values of $r$ on $(0,M)$. Then the forward and backward evolutions under the geodesic flow of any vector in the corresponding Birkhoff annulus $A$ meet $A$ again.
\end{lemma}

\begin{proof}
 For $u\in A$ we denote by $\gamma_u:\R\to \Oo$ the geodesic such that $\dot\gamma_u = u$. If $\gamma_u$ is a meridian the statement holds. For the value of the Clairaut function, we have
 \begin{align*}
  |K(u)| = r(s_0) |\cos\beta(u)| < r(s_0) \, .
 \end{align*}
 By Lemma~\ref{lemma:geodesics_surfacerev}, the geodesic $\gamma_u$ can exhibit two types of behaviour if it is not a meridian. The case of $\gamma_u$ being an asymptotic geodesic cannot occur. Otherwise its forward and backward evolutions would have to be asymptotic to equators $P_{s^\pm}$ with $r(s^\pm) = |K(u)| < r(s_0)$ in contradiction to the assumption that $P_{s_0}$ is the equator of minimal length. If $\gamma_u$ is an oscillating geodesic, it is confined to the strip between two parallels $P_{s_1}$ and $P_{s_2}$ which it alternately touches infinitely many times. Consequently, $\gamma_u$ has to intersect the equator infinitely many times, alternately at angles $\pm\beta(u)$.
\end{proof}

In fact, we can express the flow saturation of $A$ as follows:

\begin{lemma}\label{lemma:flowsaturationset}\cite[Lemma~2.2]{abbondandolo18}
 Let $s_0 \in (0,M)$ be critical point of $r$ such that $r(s_0)$ is the minimal critical value of $r$ on $(0,M)$ and let $A$ be the open Birkhoff annulus associated to the positively oriented equator $P_{s_0}$. Then
 \begin{align*}
  \bigcup_{t\in\R} \phi^t(A) = \{u \in T^1S \,|\, |K(u)|<r(s_0) \} \, .
 \end{align*}
\end{lemma}

Lemma~\ref{lemma:birkhoffannulusevolutions} guarantees that the maps given in the following definition are well-defined.

\begin{defn}
 Let $s_0 \in (0,M)$ be a critical point of $r$ such that $r(s_0)$ is the minimal critical value of $r$ on $(0,M)$ and let $A$ be the open Birkhoff annulus associated to the positively oriented equator $P_{s_0}$. We define the \textit{first return time} $\tau$ and the \textit{first return map} $\varphi$ to be the maps
 \begin{alignat*}{2}
  \tau &: A \to (0,+\infty) \, , \quad && \tau(u) := \min \{t\in(0,+\infty) \,|\, \phi^t(u) \in A \} \, , \\
  \varphi &: A \to A \, , \quad && \varphi(u) := \varphi^{\tau(u)}(u) \, .
 \end{alignat*}
\end{defn}

Again, let $s_0 \in (0,M)$ be a critical point of $r$ such that $r(s_0)$ is the minimal critical value of $r$ on $(0,M)$ and let $A$ be the open Birkhoff annulus associated to the positively oriented equator $P_{s_0}$. For points on $A$, we set 
\begin{align}\label{eq:coord_xi_eta}
\begin{split}
 \xi  &:= r(s_0)\theta \in \R/L\Z \, , \\
 \eta &:= -\cos\beta \in (-1,1) \, ,
\end{split}
\end{align}
where $L:=2\pi r(s_0)$ is the length of the equator $P_{s_0}$ and $\theta$ and $\beta$ are defined as before. Then $(\xi,\eta) \in \R/L\Z \times (-1,1)$ is a set of coordinates on $A$. The Hilbert contact form \eqref{eq:hilbertcontactform} restricted to $A$ takes the form
\begin{align*}
 \lambda := \cos\beta d\xi = -\eta d\xi
\end{align*}
with respect to these coordinates. In terms of $(\xi,\eta)$, the proof of the following result, which we will use hereafter, reduces to a short calculation.

\begin{lemma}\cite[Section~2.2]{abbondandolo17}
 Let $A \cong \R/L\Z \times (-1,1)$ be the Birkhoff annulus associated to a positively oriented equator of minimal length $L$ among all equators. The Hilbert contact form restricted to $A$, the first return time and the first return map to $A$ are related by
 \begin{align*}
  \varphi^* \lambda - \lambda = d\tau \, .
 \end{align*}
\end{lemma}

\subsection{Unit tangent bundles and lens spaces}\label{sec:tangentbundles}

For the sharp systolic inequality that we intend to prove for spindle orbifolds we need to understand the topology of the corresponding unit tangent bundles. Recall that a Riemannian orbifold is locally isometric to the quotient of a Riemannian manifold $M$ by an isometric action of a finite group $G$. This action induces an action on the unit tangent bundle of $M$, and the quotient of this action is by definition the unit tangent bundle of $M/G$. The unit tangent bundle of the orbifold can then be obtained by gluing together such local pieces. 

In case of a $2$-orbifold with isolated singularities the action of $G$ on the unit tangent bundle of $M$ is free and so the unit tangent bundle is a manifold in this case. For spindle orbifolds its topology is determined by the following lemma.

\begin{lemma}\label{lemma:lensspace}\cite[Lemma~3.1]{lange20}
 Let $\Oo$ be a $S^2(m,n)$ spindle orbifold. Then $M=T^1\Oo$ is a lens space of type $M \cong L(m+n,1)$.
\end{lemma}

We provide some further background on the statement and the proof of this lemma, because a better understanding will be required in Section \ref{chapter:topology_t1o} below.

First recall the definition of a lens space.

\begin{defn}
 Let $p$ and $q$ be integers with $\mathrm{gcd}(p,q)=1$ and consider the unit sphere $S^3 \subset \mathbb{C}^2$. The orbit space of the free action of the cyclic group $\Z_p$ on $S^3$ defined by
 \begin{align*}
  (z_1,z_2) \mapsto (e^{2\pi i p} \cdot z_1, e^{2\pi i q / p} \cdot z_2)
 \end{align*}
 is called the \textit{lens space} of type $L(p,q)$. In particular, the fundamental group of such a lens space is isomorphic to $\Z_p$.
\end{defn}

Alternatively, a lens space can be obtained by gluing together two solid tori along their boundary through a homeomorphism. To make this more precise we introduce the following notation. Let $T$ be a solid torus, i.e. $T$ is homeomorphic to $D^2 \times S^1$. We call a curve a \textit{meridian} if it is the image of $S^1 \times \{1\}$ under some homeomorphism $D^2\times S^1 \to T$. We call a curve a \textit{longitude} if it is the image of $\{1\} \times S^1$ under some homeomorphism $D^2 \times S^1 \to T$.

 With these notions at hand, we can state the following result.
 
 \begin{prop}\label{prop:gluinglensspace}Let $T_1$ and $T_2$ be solid tori and let $\psi: \partial T_1 \to \partial T_2$ be a diffeomorphism. For $i=1,2$, let $m_i$ be meridians and $l_i$ be longitudes on $T_i$. If $\psi(m_1)$ is freely homotopic to $s m_2 + r l_2$ on $\partial T_2$, the space $T_1 \cup_\psi T_2$ is a $L(r,-s)$ lens space. (see \cite[Theorem~1.3.4]{brin07} and \cite[Theorem~4.3]{jankins83}) 
 \end{prop}

Now lemma \ref{lemma:lensspace} is proved in \cite[Lemma~3.1]{lange20} by cutting $S^2(m,n)$ along an equator into two disks $D_1$ and $D_2$. Each of these two disks is a quotient of a regular disk $D^2$ by a cyclic group that acts via rotations of order $m$ and $n$, respectively. These two actions lift to diagonal actions on the unit tangent bundles $T^1 D^2 \cong D^2 \times S^1$ so that the corresponding quotients, the unit tangent bundles of $D_1$ and $D_2$, are solid tori as well. Then the proof of \cite[Lemma~3.1]{lange20} analyzes how these two solid tori are glued together and concludes the claim via an application of Proposition \ref{prop:gluinglensspace}.

\section{Systolic inequalities on orbifolds}\label{sec:surfacerevolution}

The \textit{systolic ratio} on a Riemannnian $2$-sphere $(S^2,g)$ is defined as
\begin{align*}
 \rho_\text{sys} = \frac{\ell_\text{vol}^2}{\mathrm{area}_g(S^2)} \, ,
\end{align*}
where $\ell_\text{min}$ denotes the length of the shortest closed geodesic on $(S^2,g)$. By a result of Croke this ratio is bounded from above by a constant independent of the Riemannian metric. A smooth local maximizer of this systolic ratio is Zoll, but an extremizer of the systolic ratio does not have to be smooth. In fact, the optimal upper bound is conjuctured to be attained by the singular Calabi sphere, two copies of an equilateral triangle glued together along their boundary. Nevertheless, in \cite[Theorem~1]{abbondandolo18} the authors show that on a smooth surfaces of revolution $\rho_\text{sys}$ is bounded from above by $\pi$ and that this upper bound is attained if and only if the surface is Zoll.

The above systolic ratio can also be defined for Riemannian spindle orbifolds. The authors expect that the systolic ratio is globally bounded in this case as well, but they are not aware of a proof. We point out that such a bound cannot be easily obtained from the sphere case by smoothing out the singularities, since the angles at which limits of geodesics hit the singular points cannot be controlled.

As a first possible attempt to generalize \cite[Theorem~1]{abbondandolo18} to rotationally symmetric spindle orbifolds one might conjecture the following statement: The systolic ratio $\rho_\text{sys}$ on a rotationally symmetric $S^2(m,n)$ spindle orbifold is bounded from above and the upper bound is attained if and only if the orbifold is Besse. The following example shows that this attempt is bound to fail.

Firstly, we calculate $\rho_\text{sys}$ for a rotationally symmetric $S^2(m,n)$ Besse orbifold. By Corollary~\ref{thm:tannery_orbifold}, we are given an explicit expression for the metric of any such orbifold (normalized so that the equator has length $2\pi$) which we use to compute the area of $\Oo=S^2(m,n)$. As before we denote the regular part of $\Oo$ by $U$ and compute
 \begin{align}\label{eq:besseorbifold_area}
 \begin{split}
  \mathrm{area}(\Oo) &= \int_U \sqrt{|g|} dA \\
   &= \int_0^{2\pi} \int_0^\pi \left[\frac{m+n}{2}+h(\cos r) \right] \sin r \, dr\, d\theta \\
   &= 2\pi(m+n) + \int_0^\pi h(\cos r) \sin r \, dr \\
   &= 2\pi(m+n) + \int_{-1}^1 h(v) \, dv
   = 2\pi(m+n) \ ,
 \end{split}
 \end{align}
 where the last integral vanishes as $h$ is an odd function. From Corollary~\ref{thm:tannery_orbifold}, we know that the shortest closed geodesic is the equator which has length $2\pi$. Hence we obtain
 \begin{align*}
  \rho_\text{sys}(S^2_\text{Besse}(m,n)) = \frac{2\pi}{m+n} \, .
 \end{align*}
However, we know that for a smooth surface of revolution (that is in the case $m=n=1$) $\rho_\text{sys}$ is bounded from above by $\pi$ \cite[Theorem~1]{abbondandolo18}. In particular, in the case of the round sphere, determined by the enveloping curve $\sigma: \R/2\pi\Z \to \R^2$, $\sigma(s) = (r(s),z(s)) = (\sin s,\cos s)$ as discussed in Section~\ref{sec:surfacerevolution}, the systolic ratio is $\pi$. We can change the component function $r$ on $[0,0+\varepsilon]$ and $[M-\varepsilon,M]$, i.e. in a neighbourhood of the poles, to obtain a $S^2(m,n)$ spindle orbifold as resulting surface of revolution, see Figure \ref{fig:deformedsphere}. It will differ from the standard sphere on the sets
\begin{align*}
 U_N := \bigcup\limits_{s\in[0,\varepsilon]} P_s \, , \quad U_S := \bigcup\limits_{s\in[M-\varepsilon,M]} P_s \, .
\end{align*}
By Lemma~\ref{lemma:geodesics_surfacerev}, we know that closed geodesics on this orbifold oscillate between two parallels of the same length. If we change $r$ in such a way that it remains monotonously increasing on $[0,0+\varepsilon]$ and monotonously decreasing on $[M-\varepsilon,M]$, we will not have created any additional equators. Moreover, any closed geodesic that intersects $U_N$ or $U_S$, that is, any closed geodesic which does not coincide with one on the standard sphere, still has to pass $S\setminus (U_N \cup U_S)$ twice. 

\begin{figure}[t]
  \centering
  \begin{overpic}[width=0.5\textwidth]{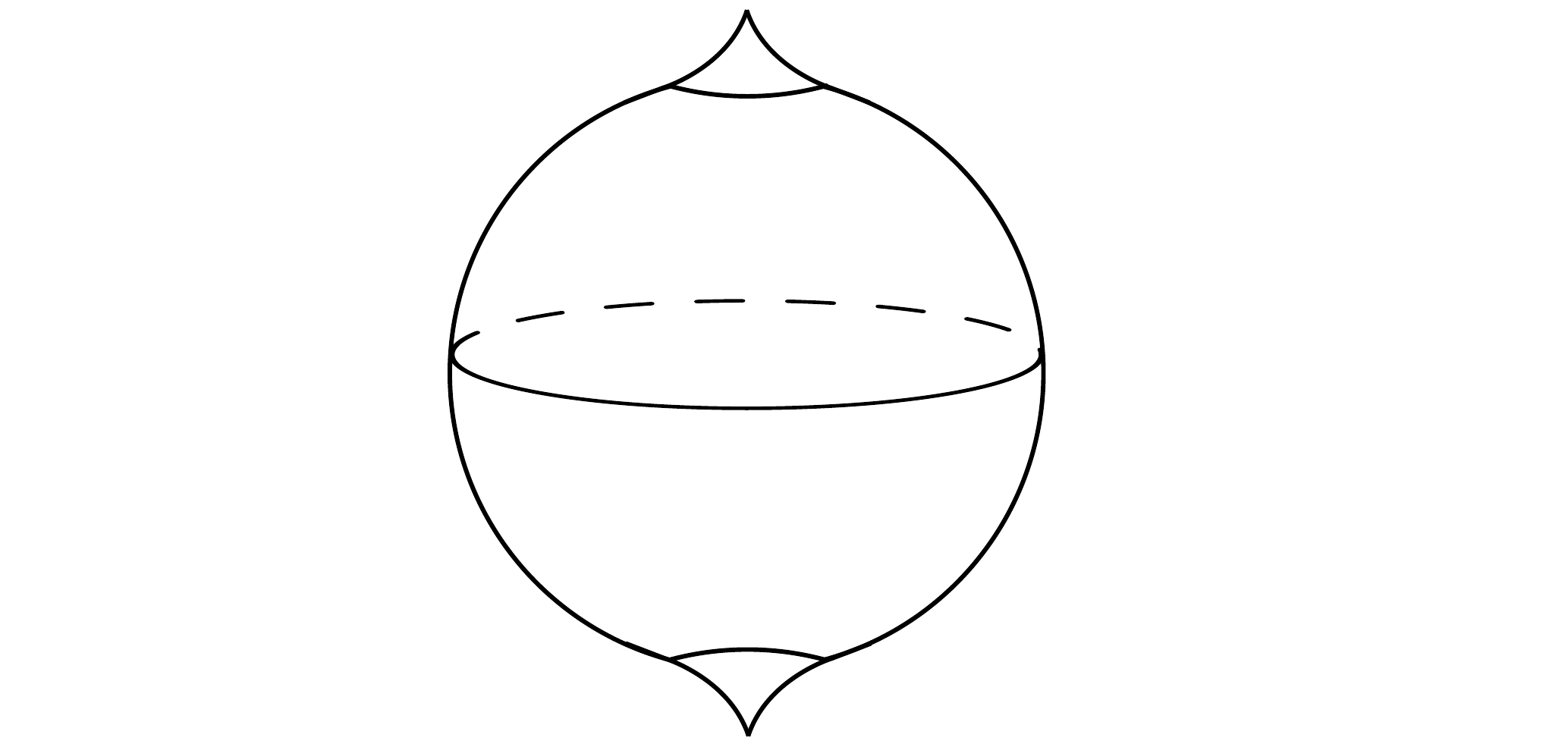}
   \put(52,45){$U_N$}
   \put(52,0.5){$U_S$}
  \end{overpic}
  \caption{Deforming $S^2$ into a spindle orbifold.}\label{fig:deformedsphere}
 \end{figure}

Consequently, the length of any closed geodesic is either $2\pi$ (the equator) or at least twice the distance of $U_N$ and $U_S$. For $\varepsilon\to 0$ that last quantity converges to $2\pi$ and the area of the spindle orbifold converges to that of a sphere. We have thus constructed a sequence of $S^2(m,n)$ spindle orbifolds for which $\rho_\text{sys}$ converges towards $\pi\geq\rho_\text{sys}(S^2_\text{Besse}(m,n))$ independently of $m$ and $n$. Hence, the naive generalization does not work.

 The above discussion can be rephrased as follows: We are interested in whether the Besse metrics maximize the systolic ratio among rotationally symmetric metrics on $S^2(m,n)$ spindle orbifolds and conclude that they fail to be global maximizers.  Even more so, following an argument similar to the one used by J. C. \'{A}lvarez Paiva and F. Balacheff in \cite{paiva14}, one can see that they even cannot be local maximizers: We start with a rotationally symmetric Besse metric on a $S^2(m,n)$ spindle orbifold. We then disturb the metric in a small neighbourhood $V$ of a parallel far away from the equator such that the area decreases. Let $\phi_t$ denote the geodesic flow on $T^1\Oo$. Before the perturbation, the equator was the shortest closed geodesic with length to $2\pi$ and all other geodesics were closed as well with a length greater than or equal $4\pi$ for $n+m \ge 3$ (see Corollary~\ref{thm:tannery_orbifold}). This implies that $\phi_{2\pi}(u) \neq u$ for all $u \in T^1V$ which remains true for sufficiently small perturbations. As such, any closed geodesic created by perturbing the metric on $U$ cannot have length shorter or equal to $2\pi$.  The equator is unchanged under our deformation and thus remains the shortest closed geodesic. Consequently, the systolic ratio of the deformed orbifold has to be larger than that of the Besse orbifold. \\
 
 In light of the discussion above, we consider a different notion of systolic ratio.
\begin{defn}
 We define the \textit{contractible systolic ratio} of a spindle orbifold $\mathcal{O}=S^2(m,n)$ to be the quantity
 \begin{align*}
  \rho_\text{contr}(\Oo) = \frac{\ell_\text{min,contr}^2}{\mathrm{area}(\Oo)} \, ,
 \end{align*}
 where $\ell_\text{min,contr}$ denotes the length of the shortest closed geodesic whose lift to the unit tangent bundle $T^1\Oo$ is contractible.
\end{defn}

With this notion we will be able to prove Theorem \ref{thm:A}. For that we need to understand, when a lift of a closed geodesic to the unit tangent bundle is contractible. This will be achieved in the following section.

\subsection{Lifting curves to the unit tangent bundle}\label{chapter:topology_t1o}

  \begin{figure}[b] 
  \centering
  \begin{overpic}[width=0.7\textwidth]{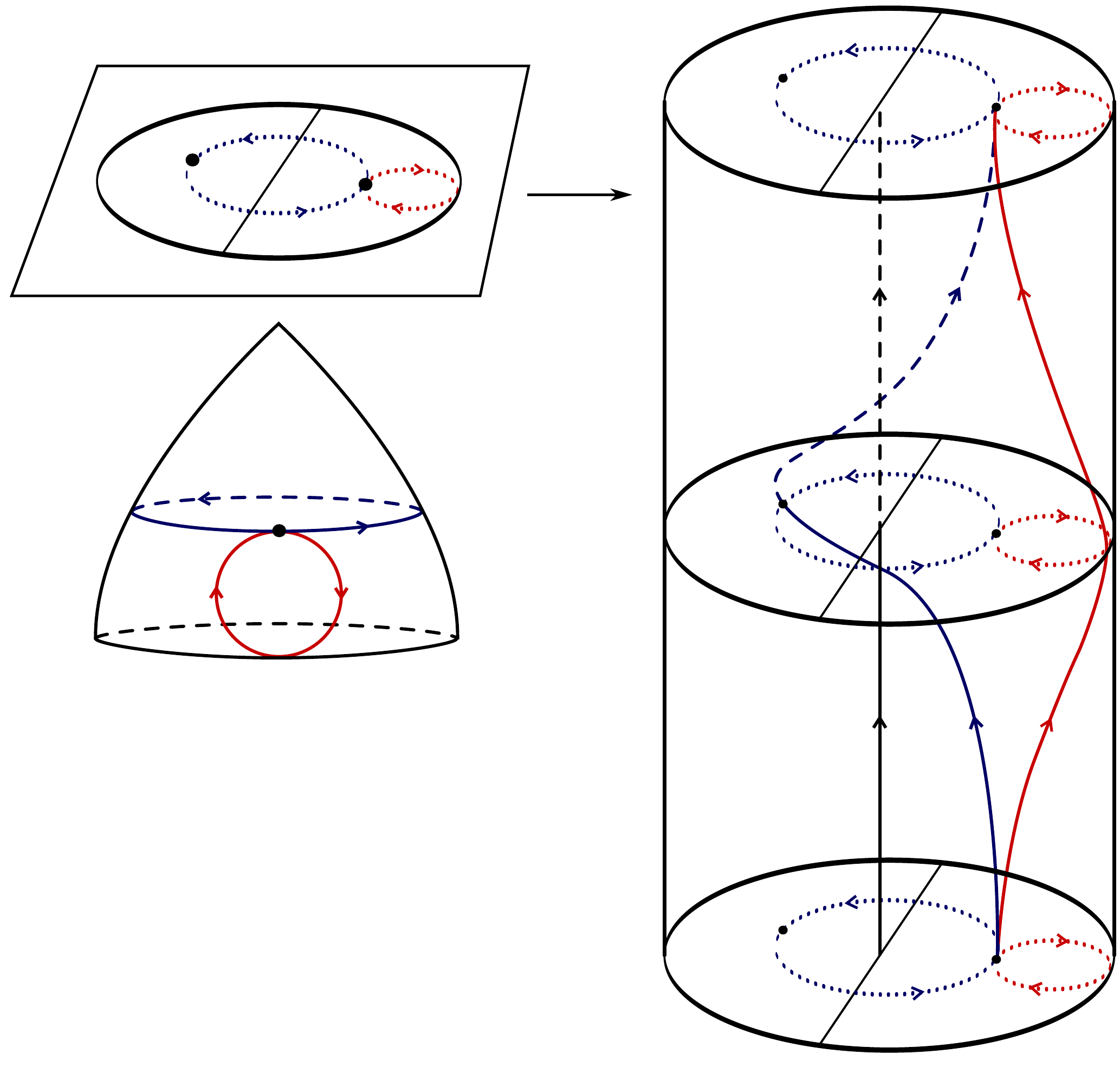}
   \put(6,48.5){$\color{figureblue} c_1$}
   \put(33.5,41.5){$\color{figurered} c_2$}
  \end{overpic}
  \caption{Exemplary simple closed curves on an orbifold and their lift to a chart of the universal covering of the unit tangent bundle (in the case of a singularity of degree $2$).}\label{fig:covering}
 \end{figure}

Let $\Orb$ be a spindle orbifold of type $S^2(m,n)$. We call a closed curve ${\gamma:[0,1]\to \Oo}$ \emph{regular} if it is smooth (in the sense that it admits smooth lifts to manifold charts and that the velocities at the boundary match up smoothly) and has nowhere vanishing velocity. A regular curve can be lifted to the unit tangent bundle of $\Orb$ as the curve $(\gamma(t),\dot\gamma(t) \|\dot\gamma(t)\|^{-1})$. We will need to understand under which conditions on $\gamma$ this lift is contractible. Recall from Lemma \ref{lemma:lensspace} that the unit tangent bundle of $\Orb$ is a lens space of type $L(m+n,1)$ and hence its fundamental group isomorphic to $\Z_{m+n}$. A priorily the lift of $\gamma$ represents any element in this fundamental group. Clearly, the homotopy class of this lift is invariant under isotopic deformations of $\gamma$. Therefore, to settle the general problem, we can assume without loss of generality that the curve $\gamma$ stays in the regular part of $\Orb$. Any such regular curve in the regular part is in turn isotopic to a concatenation of a finite number of copies of two specific regular curves, $c_1$, homotopic to an equator of $\Orb$, and $c_2$, a regular simple closed curve in the regular part of $\Orb$ which is contractible within this regular part. Hence, it suffices to identify the homotopy classes of the lifts of $c_1$ and $c_2$ in order to answer the general question. 

We can assume that the two curves $c_1$ and $c_2$ are contained in the complement $D_1$ of a small open ball around the singular point of order $n$. The disk $D_1$ is a quotient of a disk $D^2$ by a cyclic group $\Z_m$ that acts via rotations on $D^2$. In Section \ref{sec:tangentbundles} we have seen that the unit tangent bundles of $D^2$ and $D_1$ are both tori and that $\Z_m$ acts diagonally on $T^1 D^2 = D^2 \times S^1$ with quotient $T_1:=T^1 D_1$. Moreover, there we have seen that the unit tangent bundle of $\Orb$ can be obtained by gluing together $T_1$ with its complementary torus along there common boundary. The critical observation for our argument now is that the central fiber of the torus $T_1$ generates the fundamental group of the resulting lens space $T^1 \Orb$ which can be seen by an application of the Seifert-van-Kampen theorem.

The curve $c_2$ lifts to a simple closed curve in $D^2$. Hence, its lift to $T^1 D^2$ is homotopic to the central fiber of this torus, and its lift $T_1$ is homotopic to $m$ times the central fiber of the latter torus, cf. Figure \ref{fig:covering}. The curve $c_1$ does not lift to a closed curve in $D^2$, but its lift to $T_1$ is homotopic to the central fiber,  cf. Figure \ref{fig:covering}. In particular, we see that the lift of $c_1$ generates the fundamental group of $T^1 \Orb$, whereas the lift of $c_2$ becomes contractible after being iterated $k$-times, where

\begin{align*}
 k \cdot m = \mathrm{lcm}(m,m+n) = \frac{m(m+n)}{\mathrm{gcd}(m,m+n)}
\end{align*}
from which we get
\begin{align*}
 k = \frac{m+n}{\mathrm{gcd}(m,m+n)} = \frac{m+n}{\mathrm{gcd}(m,n)} \, .
\end{align*}

As this discussion is valid for any simple closed regular curve in the regular part of the orbifold we have just shown the following.

\begin{prop}\label{prop:contractible}
 Let $\Oo=S^2(m,n)$ be a spindle orbifold and let $c$ be a simple closed regular curve in the regular part of $\Oo$. Then the lift of $c$ to the unit tangent bundle $T^1\Oo$ is non-contractible. Furthermore, the following holds
 \begin{enumerate}[i)]
 \item If $c$ encloses one of the singularities (and therefore both), it becomes contractible after being iterated by any non-zero integer multiple of $(m+n)$-times.
 \item If $c$ does not enclose the singularities, it becomes contractible after being iterated by any non-zero integer multiple of $k$-times, for
 \begin{align*}
  k = \frac{m+n}{\mathrm{gcd}(m,n)} \, .
 \end{align*}
 \end{enumerate}
\end{prop}


With regard to rotationally symmetric spindle orbifolds (compare Section~\ref{sec:surfacerevolution}) we conclude the following.

\begin{cor} \label{cor:contractible_windingnr}
 Let $\Oo=S^2(m,n)$ be a rotationally symmetric spindle orbifold. Then a closed geodesic $\gamma$ in the regular part of $\Oo$ has a contractible lift in $T^1\Oo$ if and only if it winds
 \begin{align*}
  |W(\gamma)| = k (m+n) \, ,
 \end{align*}
 many times around the singular points, where $k$ is a positive integer. In particular, any equator is non-contractible in $T^1\Oo$, and it becomes contractible after being iterated $k(m+n)$ times, where $k$ is a non-zero integer. Moreover, if $\Oo$ is in addition Besse for every other closed geodesic we have:
 \begin{enumerate}[i)]\label{prop:iterates_contractible} \setcounter{enumi}{1}
 \item If $m+n$ is odd, the closed geodesic and all its iterates are contractible in $T^1\Oo$.
 \item If $m+n$ is even, the closed geodesic is non-contractible in $T^1\Oo$. It becomes contractible when iterated $2k$ times, where $k$ is a non-zero integer.
 \end{enumerate}   

\end{cor}

\begin{proof}
 We note that a geodesic is necessarily a regular curve and set $w:=|W(\gamma)|$. From Lemma~\ref{lemma:thetaderivative}, we know that $|\dot\theta|>0$. Consequently, a closed geodesic $\gamma$ cannot have any contractible subloop. It is therefore regularly homotopic to $\alpha^w$ for a simple closed curve $\alpha$ with winding number $1$ (and where by regularly homotopic we mean homotopic such that every curve in the homotopy is regular). Any regular homotopy on $\Oo$ corresponds to a homotopy in $T^1\Oo$. Thus, the lift of $\gamma$ to $T^1\Oo$ is homotopic to the lift of $\alpha^w$. The statement then follows from Proposition~\ref{prop:contractible}.

 The statements about equators and closed geodesics other than meridians follow from Corollary~\ref{cor:contractible_windingnr} and Corollary~\ref{thm:tannery_orbifold}. The statement about meridians follows from the fact, that they can be approximated by closed geodesics (that are not meridians).
\end{proof}

\begin{remark}\label{rmk:windingnumber_homotopyclass}
The above generalizes to the following statement. Let $\Oo=S^2(m,n)$ be a rotationally symmetric spindle orbifold and $\gamma$ be a closed geodesic in the regular part of $\Oo$. If $W(\gamma)=k$, then the lift of $\gamma$ to $T^1\Oo$ represents an element in the fundamental group which lies in a subgroup corresponding to the subgroup generated by $k$ in $\Z_{n+m}$.
\end{remark}

At this point we also record the following consequence.

\begin{cor}\label{cor:contractible_length_spectrum}
Let $\Oo=S^2(m,n)$ be a rotationally symmetric Besse spindle orbifold. Then for any prime geodesic the minimal length of an iterate that has a contractible lift to the unit tangent bundle $T^1\Oo$ is $2(m+n)\pi$.
\end{cor}

\begin{proof}
 By Corollary~\ref{thm:tannery_orbifold}, we know that the equator has length $2\pi$ and every other geodesic has length $2\frac{m+n}{2-\alpha}\pi$, where again
 \begin{align*}
  \alpha = \begin{cases} 0 \, , & m+n \text{ even} \\ 1 \, , & m+n \text{ odd} \end{cases} \, .
 \end{align*}  
 Furthermore, by Corollary~\ref{cor:contractible_windingnr}, the equator has to be iterated at least $(m+n)$-times to be contractible in $T^1\Oo$; any other geodesic has to be iterated at least $(2-\alpha)$-times and the claim follows.
\end{proof}

\subsection{Inequalities for the contractible systolic ratio}

Now we readily obtain one implication of Theorem \ref{thm:A} in the following proposition.

\begin{prop}\label{prop:systolic_ratio_besse}
 The contractible systolic ratio of a rotationally symmetric $S^2(m,n)$ Besse spindle is
 \begin{align*}
  \rho_\text{contr}(S^2_\text{Besse}(m,n)) = 2(n+m)\pi \, .
 \end{align*}
\end{prop}

\begin{proof}
 In our discussion at the beginning of this section, we already calculated the area of a rotationally symmetric Besse spindle orbifold $S^2(m,n)$ to be $\mathrm{area}(\Oo) = 2(m+n)\pi$. Moreover, from Corollary~\ref{cor:contractible_length_spectrum} we know that the length of any closed geodesic that lifts to a contractible curve in $T^1\Oo$ is $2(m+n)\pi$ and this proves the claim.
\end{proof}

Before proving the converse in the following section let us discuss some related remarks and variants of Theorem \ref{thm:A}.

The geodesic flow on $T^1\Oo$ can be seen as a Reeb flow. According to Lemma~\ref{lemma:lensspace} $T^1\Oo$ is diffeomorphic to $L(m+n,1)$ and so it is covered by $S^3$. In particular, the Reeb flow on $T^1\Oo$ lifts to a Reeb flow on $S^3$. By \cite{lange20} this lifted Reeb flow is Zoll. In the present rotationally symmetric setting the same conclusion follows more easily from Corallary~\ref{cor:contractible_length_spectrum}
 
 We recall a useful relation between the surface area and the contact volume of a surface which extends to the case of $2$-orbifolds with isolated singularities:
\begin{align}\label{eq:areacontactvolume}
 \mathrm{area}(\Oo) = \frac{1}{2\pi} \vol(T^1\Oo) \, .
\end{align}
 For manifolds the equality is proven for instance in \cite[Proposition~3.7]{abbondandolo17} by explicit calculation in isothermal coordinates. It remains valid under removing sets of measure zero which gives the same equality in the case of orbifolds with isolated singularities. We will use it in the proofs of our main results and we can also use it to relate the contractible systolic ratio to the standard systolic ratio for contact forms on $S^3$: The standard systolic ratio for $S^3$ endowed with a contact form $\alpha$ is defined as
 \begin{align*}
  \rho_\text{sys}(S^3,\alpha) = \frac{T_\text{min}(\alpha)^2}{\vol(S^3,\alpha\wedge d\alpha)} \, ,
 \end{align*}
 where $T_\text{min}(\alpha)$ denotes the minimal period of all periodic orbits. Consider now the contact structure $\alpha_\phi$ on $S^3$ induced by the lift of the geodesic flow $\phi$ as described before. Since the contractible orbits on $T^1\Oo$ lift to closed orbits on $S^3$, we see that $T_\text{min}(\alpha_\phi) = \ell_\text{min, contr}$. Taking into account \eqref{eq:areacontactvolume} and the degree of the covering $T^1\Oo = L(m+n,1) \to S^3$, we get
 \begin{align}\label{eq:systolicratio_contr_contact}
  \rho_{\text{contr}}(S^2(m,n)) = 2\pi(m+n) \rho_\text{sys}(S^3,\alpha_\phi)
 \end{align}
 as a relation between both systolic ratios. Since standard systolic ratio for any Zoll contact form on $S^3$ is $1$ \cite[Theorem~1]{abbondandolo18_2}, we recover the relation
  \begin{align*}
   \rho_{\text{contr}}(S^2(m,n)) = 2\pi(m+n) \, .
  \end{align*}

The relation between Besse metrics on the orbifold and Zoll contact forms on $S^3$ also yields the following.

\begin{remark}
 Zoll contact forms on $S^3$ constitute strict local maximizers of the systolic ratio on $S^3$ in a $C^3$-topology \cite[Theorem~1]{abbondandolo18_2}. Moreover, relation \eqref{eq:systolicratio_contr_contact} tells us that the contractible systolic ratio on rotationally symmetric spindle orbifolds and the standard sytolic ratio of the contact form obtained by lifting the geodesis flow on $T^1\Oo$ to $S^3$ coincide up to a constant factor. Thus, rotationally symmetric Besse metrics on spindle orbifolds  are strict local maximizers of the systolic ratio in $C^3$-topology among all $S^2(m,n)$ spindle orbifolds (not just the rotationally symmetric ones).
\end{remark}

A generalization of the contractible systolic ratio was suggested to us by P.A.S. Salom\~{a}o:

\begin{defn}
 On a spindle orbifold $\mathcal{O}=S^2(m,n)$ and a divisor $k$ of $(m+n)$ we define the quantity
 \begin{align*}
  \rho_\text{sys,k}(\Oo) = \frac{\ell_\text{min,k}^2}{\mathrm{area}(\Oo)} \, ,
 \end{align*}
 where $\ell_\text{min,k}$ denotes the length of the shortest closed geodesic whose lift to the unit tangent bundle $T^1\Oo$ lies in the subgroup of $\pi_1(T^1\Oo)$ of order $k$.
\end{defn}
Note how for $k=1$ this definition coincides with our notion of the contractible systolic ratio.

\begin{remark}\label{rmk:no_local_maximizer_of_rho_sysk}
 Unless $k=m+n$ (as in Theorem A) or $k=2$ for $m+n$ even (as in Theorem B) the Besse metrics even fail to locally maximize the systolic ratio $\rho_{sys,k}$ among rotationally symmetric metrics. This can be seen by an argument similar to the one at the beginning of this section which was used to show that Besse metrics do not locally maximize the classical systolic ratio among rotationally symmetric metrics. In fact, for $k$ not like in the specific cases stated above, we know that on a rotationally symmetric Besse orbifold the $k$-th iterate of the equator is the only closed geodesic whose lift represents an element in the subgroup of $\pi_1(T^1\Oo)$ of order $(n+m)/k$; it has length $2\pi k$ and any other geodesic has length $2\frac{m+n}{2-\alpha}\pi$ (recall Proposition \ref{thm:tannery_orbifold} and Remark \ref{rmk:windingnumber_homotopyclass}). We now deform the orbifold in a neighbourhood $V$ of a parallel away from the equator like in the above mentioned argument. In doing so, we do not create new closed geodesics contributing to $\rho_{sys,k}$ as, before the deformation, we have $\phi_{2\pi k}(u) \neq u$ for all $u\in T^1 V$ (where again $\phi$ denotes the geodesic flow) which remains true also after a sufficiently small deformation.
\end{remark}

\begin{remark}
  Again, we can consider the geodesic flow as a Reeb flow on the unit tangent bundle. Then for a suitable covering of $L(m+n,1)$, the systolic ratio $\rho_{\text{sys,k}}$ coincides with the standard systolic ratio of the lifted Reeb flow up to a multiplicative constant, similarly to \eqref{eq:systolicratio_contr_contact}. For Besse spindle orbifolds the only such lifts that are Zoll are the lifts to the universal covering $S^3$ and, if $n+m$ is even, to the $\frac{m+n}{2}$-fold covering $L(2,1)$.
\end{remark}

For the systolic ratio $\rho_{\text{sys},k}$ we will prove Theorem B in the same way as Theorem A.

\begin{remark}
 For $m=n=1$ a $S^2(m,n)$-spindle orbifold is just a smooth sphere. Moreover, we know that on rotationally symmetric spheres every closed geodesic is a simple closed curve. As such its lift will be non-contractible by Corollary~\ref{cor:contractible_windingnr}. Consequently, for a Zoll sphere (remember that on $S^2$ Besse implies Zoll) $\rho_{\text{contr},2}$ coincides with the standard systolic ratio considered in \cite{abbondandolo18}.
\end{remark}

\section{The generating function}\label{sec:generatingfunction}

Let $\Oo$ be a rotationally symmetric spindle orbifold as discussed in Section~\ref{sec:surfacerevolution} and let $A$ be an open Birkhoff annulus associated to an equator of minimal length (see Section~\ref{ssec_returnmapreturntime}). Recall that the minimality of the equator ensures that every geodesic emanating from $A$ is an oscillating geodesic in terms of the dichotomy given by Lemma~\ref{lemma:geodesics_surfacerev}. The first return time and first return map (see Section~\ref{ssec_returnmapreturntime}) depend on the length of the geodesic arc that comprises one oscillation. This information is encoded in a single map $F$, in terms of which $\tau$ and $\varphi$ take a particularly nice form with respect to the coordinates \eqref{eq:coord_xi_eta}, which is given by the following lemma.

\begin{lemma}\label{lemma:gen_fct_return_mapandtime}
 Let $A \cong \R/L\Z \times (-1,1)$ be the Birkhoff annulus associated to a positively oriented equator of minimal length $L$ among all equators. Then the first return map $\varphi: A \to A$ and the first return time $\tau: A \to (0,+\infty)$ to $A$ take the form
 \begin{align*}
  \tau(\xi,\eta) &= F(\eta) - \eta F'(\eta) \, , \\
  \varphi(\xi,\eta) &= (\xi + F'(\eta) + \alpha {\textstyle \frac L2}, \eta) \, ,
 \end{align*}
 where $F:(-1,1)\to\R$ is an even smooth function and where we have used the notation
 \begin{align*}
   \alpha = \begin{cases} 0 \, , \quad & m+n \text{ even} \\ 1 \, , \quad & m+n \text{ odd} \end{cases} \, .
 \end{align*}
\end{lemma}

We will refer to this function $F$ as the \textit{generating function} of the first return map $\varphi$.

\begin{proof}
 We adapt the proof of \cite[Lemma~3.1]{abbondandolo18} to the orbifold case. First, we note that the restriction of the Clairaut function to $A$ is given by
 \begin{align*}
  K(u)|_A = r(s_0)\cos\beta = -r(s_0)\eta \, .
 \end{align*}
 Consequently, the second component of $(\xi,\eta)$ is preserved by the first return map $\varphi$. Furthermore, the rotational symmetry of $S$ implies that $\varphi$ commutes with translations in the $\xi$-coordinate. Thus, $\varphi$ takes the form
 \begin{align*}
  \varphi(\xi,\eta) = (\xi + f(\eta) + \alpha \frac{L}{2}, \eta) \quad \forall (\xi,\eta) \in \R/L\Z \times (-1,1)
 \end{align*}
 for a smooth function $f:(-1,1)\to\R$, where we do not absorb the summand $\alpha\frac{L}{2}$ into $f$ for reasons that will become clear shortly hereafter. The function $f$ is only defined modulo $L$. We can normalize $f$ by considering the behaviour of meridians. From the discussion in last paragraph of Section \ref{sub:rotationally_symmetric_metrics}, we see that for $(m+n)$ even the meridian starting at any point $x$ of $A$ has its point of first recurrence at $x$, whereas for $(m+n)$ odd, the point of first recurrence is the antipode of $x$. In terms of the first return map $\varphi$, this implies
 \begin{align*}
  \varphi((\xi,0)) = \begin{cases} (\xi,0) \, , \quad & m+n \text{ even} \\ (\xi + \frac L2,0) \, , \quad & m+n \text{ odd} \end{cases} \quad \forall \xi\in \R/L\Z \, .
 \end{align*}
 Consequently, we can normalize $f$ by requiring $f(0)=0$. The symmetry of $S$ with respect to reflection through planes containing the $z$-axis implies that 
 \begin{align*}
  \varphi(\xi,-\eta)=(\xi + f(-\eta) + \alpha\frac L2, -\eta) = (\xi - f(\eta) - \alpha\frac L2, -\eta)
 \end{align*}
 for all $(\xi,\eta) \in \R/L\Z \times (-1,1)$. This implies that
 \begin{align*}
  f(-\eta) = -f(\eta) + kL
 \end{align*}
 for a $k\in\Z$. The normalization $f(0)=0$ now forces $k=0$ and thus $f$ is an odd function (independently of $\alpha$, which is why we chose to break of the $\alpha\frac L2$ summand earlier).\\[0.5ex]
 Now, let $F:(-1,1)\to\R$ be a primitive of $f$. $F$ has to be even as $f$ is odd. We have thus shown that the first return map $\varphi$ can be expressed as given in the statement. Nevertheless, $F$ is not yet determined uniquely but only up to the constant of integration. We notice that
 \begin{align*}
 d\tau = \varphi^*\lambda - \lambda &= - \eta d\left( \xi + F'(\eta) + \alpha \frac L2 \right) + \eta d\xi = -\eta F''(\eta) d\eta = d \left( F(\eta) - \eta F'(\eta) \right) \, ,
 \end{align*}
 independently of $\alpha$. This means, that in both cases, $\tau$ and $F(\eta) - \eta F'(\eta)$ differ by a constant and we can normalize $F$ such that
 \begin{align*}
  \tau(\xi,\eta) = F(\eta) - \eta F'(\eta) \quad \forall (\xi,\eta) \in \R/L\Z \times (-1,1)
 \end{align*}
 which yields the desired identity.
\end{proof}

 By Lemma~\ref{lemma:gen_fct_return_mapandtime} we can relate critical points of $F$ and closed geodesics. Let $\eta_0$ be a critical point of $F$ and denote by $\mu:=F(\eta_0)$ the corresponding critical value. Then we have
 \begin{align*}
  \varphi(\xi,\eta_0) = (\xi + \alpha\frac{L}{2},\eta_0) \, .
 \end{align*}
 This means that for all $\xi \in \R/L\Z$ the points $(\xi,\eta_0)$ are fixed points of $\varphi$ in the case $m+n$ odd ($\alpha=0$) or of $\varphi^2$ in the case $m+n$ even ($\alpha=1)$. These correspond to closed geodesics $\gamma$ of length $(1+\alpha)\tau(\xi,\eta_0)$. Consequently, by the expression for the return time given in the lemma, we have found:

\begin{cor}\label{rmk:critpointsgeneratingfct}
 Critical points $\eta_0$ of $F$ are in one-to-one correspondence with closed geodesics starting in $A$ of length $(1+\alpha)F(\eta_0)$.
\end{cor}

\section{Properties of the generating function}\label{sec:generatingfunction_prop}

In the following we analyze properties of the generating function used in \cite{abbondandolo18} in our present orbifold setting. Let $u$ be a unit tangent vector in $A$ given by $(\xi,\eta) \in \R/L\Z \times (-1,1)$ with $\eta \neq 0$. Then the geodesic with starting vector $u$ is not a meridian and therefore does not intersect the $z$-axis. Consequently, for any $t>0$, the winding number of $\gamma_u|_{[0,t]}$ with respect to the $z$-axis is well-defined. We associate to every $u$ a winding number $W(u)$ given by
\begin{align*}
 W(u) := \frac{\theta(\tau(u))-\theta(0)}{2\pi} \, ,
\end{align*}
where $\theta:\R\to\R$ is a continuous function such that
\begin{align*}
 \gamma_u(t) = (r(s(t))\cos\theta(t), r(s(t))\sin\theta(t),z(s(t))) \quad \forall t\in\R \, .
\end{align*}
 (In fact, this function $\theta$ coincides with the $\theta$-coordinate of $\gamma_u(t)$ in the coordinate system \eqref{eq:metricrotation} up to integer multiples of $2\pi$ which justifies our abuse of notation). Like the first return map $\varphi(u)$, the winding number $W(u)$ does not depend on the $\xi$-coordinate of $u$ and consequently we denote it by $W(\eta)$.

The following lemma tells us that the winding number and the generating function are not independent of each other.

\begin{lemma}\label{lemma:windingnumber_generatingfct}
 The winding number and the generating function $F$ are related by the following identities:
 \begin{align*}
  F'(\eta) = \begin{cases} LW(\eta) - \frac{m+n}{2}L & \forall \eta\in (-1,0) \\ LW(\eta) + \frac{m+n}{2}L & \forall \eta\in (0,1) \end{cases} \, .
 \end{align*}
\end{lemma}

\begin{proof}
 Both the cases of $(m+n)$ being even or odd work similarly to the proof of Lemma 4.1 in \cite{abbondandolo18}. We state the odd case for completeness.

 Let $\eta \in (-1,0)$. Following the discussion about the behaviour of meridians on rotationally symmetric spindle orbifolds at the end of Section~\ref{sub:rotationally_symmetric_metrics}, we conclude following: a geodesic, which is close to a meridian and that emanates from a parallel, turns approximately $\Delta\theta = k\pi$ in the angle-$\theta$-coordinate before returning to the same parallel after passing near a singularity of order $k$. More specifically, if we consider a sequence of curves converging to a meridian, $\Delta\theta$ converges to $k\pi$. In our case $m+n$ odd, this implies the following:
 The winding number $W(\eta)$ of the geodesic emanating from some point on the equator with starting velocity $\eta$ on the interval $[0,\tau(\eta)]$ has the following behaviour in the limit for negative $\eta$ close to zero:
 \begin{align*}
  \lim\limits_{\eta \nearrow 0} W(\eta) = \frac{n+m}{2} \, ,
 \end{align*}
 In fact, assume without loss of generality that $m$ is odd and $n$ is even. A geodesic arc starting on $(\xi,\eta)\in A$ for negative $\eta$ close to zero can have two behaviours: As a first alternative, it can first pass close to the $\Z_m$ singularity which means that it will reach the equator near $(\xi+\frac{L}{2},-\eta) \notin A$ having collected approximately $\Delta\theta=m\pi$. It then passes close to the $\Z_n$ singularity and meets the equator close to $(\xi+\frac L2,\eta) \in A$ having collected approximately $\Delta\theta = (m+n)\pi$ with respect to the starting point. As a second alternative, the geodesic may pass the singularities in the order $\Z_n$, $\Z_m$. Still, at its point of first recurrence to $A$, it will have $\Delta\theta$ close to $(m+n)\pi$ as well.

Now, for $(m+n)$ odd, the first component of $\varphi(\xi,\eta)$ was given by $\xi+F'(\eta)+L/2$ on $\R/L\Z$. This implies that the angle $\theta$ (taking values in $\R$) at time $\tau(\eta)$ is given by $2\pi/L\cdot(\xi+F'(\eta)+L/2)$ up to an integer multiple of $2\pi$. Combining these considerations with the definition of the winding number, we get
 \begin{align*}
  W(\eta) = \frac{\theta(\tau(\eta))-\theta(0)}{2\pi} = \frac{1}{2\pi}\left( \frac{2\pi}{L}(\xi+F'(\eta)+\frac{L}{2})+2\pi k - \frac{2\pi}{L}\xi \right)
 \end{align*}
 for some $k\in\N$ and all $\eta\in(-1,0)$ or equivalently
 \begin{align*}
  F'(\eta) = LW(\eta) - \frac{L}{2} + kL \quad \forall \eta\in(-1,0) \ .
 \end{align*}
 By taking the limit $\eta\nearrow 0$ and using the limit behaviour of $W(\eta)$ as stated above, as well as the fact that $F'$ is smooth with $F'(0)=0$, we get $k = \frac{1}{2}-\frac{m+n}{2}$ and subsequently
 \begin{align*}
  F'(\eta) = LW(\eta) - \frac{L}{2} + \left( \frac{1}{2}-\frac{m+n}{2} \right) L = LW(\eta) - \frac{m+n}{2}L  \quad \forall\eta\in(-1,0) \ .
 \end{align*}
 The corresponding identity for $\eta\in(0,1)$ follows from the fact that both $W$ and $F'$ are odd functions in $\eta$.
\end{proof}

 In Corollary~\ref{rmk:critpointsgeneratingfct} we observed that critical points $\eta_0$ of $F$ are in one-to-one correspondence with closed geodesics starting in $A$ of length $(1+\alpha)F(\eta_0)$. By Lemma~\ref{lemma:windingnumber_generatingfct} such closed geodesics satisfy $W(\gamma) = (1+\alpha)\frac{m+n}{2}$. Together with Corollary~\ref{cor:contractible_windingnr} we obtain

\begin{cor}\label{rmk:critpointsgeneratingfct_contractiblegeodesic}
 We have the following one-to-one correspondence:
 \begin{align*}
  \text{Critical point $\eta_0$ of $F$} \ & \stackrel{1:1}{\leftrightarrow}  \ \parbox{11cm}{Closed geodesic with length $(1+\alpha)F(\eta_0)$ starting in $A$ whose lift to $T^1\Oo$ lies in the subgroup of $\pi_1(T^1\Oo)$ of order  $\frac{2}{1+\alpha}$}
\intertext{and in particular}
  \text{Critical point $\eta_0$ of $F$} \ & \stackrel{1:1}{\leftrightarrow} \  \parbox{11cm}{Closed geodesic starting in $A$ with contractible lift to $T^1\Oo$ and length $2F(\eta_0)$} \, .
 \end{align*}
\end{cor}

Next, we show that the generating function can be expressed by an integral formula which allows to extend it continuously to the closed intervall $[-1,1]$. We define the sets
\begin{align*}
 \Omega_\kappa &:= \{ (\beta,s) \in {\textstyle \left( - \frac \pi 2, \frac \pi 2 \right) } \times (0,M) \,|\, K(\beta,s) > \kappa \} \, , \\
 \Gamma &:= \{ (\beta,s) \in {\textstyle \left( - \frac \pi 2, \frac \pi 2 \right) } \times (0,M) \,|\, K(\beta,s) \ge r(s_0) \} \, .
\end{align*}

\begin{lemma}\label{lemma:generatingfct_integral}
 The generating function $F$ can be expressed by the identity
 \begin{align*}
  F(\eta) = \int_{\Omega_{\kappa(\eta)}} \cos\beta\, d\beta \wedge ds + \frac{m+n}{2} L |\eta| \quad \forall \eta \in (-1,1) \, , 
 \end{align*}
 where $\kappa(\eta) := |K(u)| = r(s_0)|\eta|$. In particular, $F$ can be extended to $[-1,1]$ by setting
 \begin{align*}
  F(-1) = F(1) = \int_\Gamma \cos\beta\, d\beta \wedge ds + \frac{m+n}{2} L \, .
 \end{align*}
\end{lemma}

\begin{proof}
 The proof of the corresponding statement for the smooth case given in \cite[Lemma~4.3]{abbondandolo18} carries over by taking into account Lemma~\ref{lemma:windingnumber_generatingfct}.
\end{proof}

From the above lemma, we conclude a bound on the value of the generating function.

\begin{cor}\label{lemma:generating_function_bound}
 The generating function $F$ is bounded from below by
 \begin{align*}
  F(\eta) > \frac{m+n}{2} L |\eta| \, .
 \end{align*}
\end{cor}

\begin{proof}
 We consider the formula for $F$ given in Lemma~\ref{lemma:generatingfct_integral} and notice that $\cos\beta$ is positive on $\Omega_\kappa$. Consequently, the integral term is positive and the bound on $F$ follows.
\end{proof}

The proof of our main result will rely on the fact that we can express the contact volume of the unit tangent bundle in terms of the generating function as follows.

\begin{lemma}\label{lemma:contactvolume}
The contact volume can be expressed as
 \begin{align*}
  \vol(T^1\Oo) = 4L \int_0^1 F(\eta)\, d\eta - (m+n)L^2 + \int_\Gamma (4\pi r(s) - 2L \cos\beta)\, d\beta \wedge ds \, .
 \end{align*}
\end{lemma}

To prove this lemma, we first show how the contact volume of the flow saturation of $A$ can be expressed as an integral formula involving the first return time.

\begin{lemma}\label{lemma:flowsaturationvolume}
  Let $A \cong \R/L\Z \times (-1,1)$ be the Birkhoff annulus associated to a positively oriented equator of minimal length $L$ among all equators. Denote by
 \begin{align}
  \tilde A = \bigcup_{t\in\R} \phi^t(A)
 \end{align}
 the flow saturation of $A$. Then the contact volume of $\tilde A$ is given by
 \begin{align*}
  \vol(\tilde A) = \int_a \tau \omega \, .
 \end{align*}
\end{lemma}

\begin{proof}
 The proof follows a standard argument which we adapted from \cite[Lemma~3.7]{abbondandolo18_2}. We define a function
 \begin{align*}
  \psi: [0,1] \times A \to \tilde A \, , \quad \psi(s,u) = \phi^{s\tau(u)}(u) \, ,
 \end{align*}
 where $\psi^t$ denotes the geodesic flow (that is the Reeb flow of the Hilbert contact form $\alpha$). The map $\psi$ is a bijection from $[0,1) \times A$ to $\tilde A$. Consequently, we have
 \begin{align*}
  \vol(\tilde A) = \int\limits_{[0,1]\times A} \psi^*(\alpha \wedge d\alpha) \, .
 \end{align*}
 Now the same explicit calculation as in \cite[Lemma~3.7]{abbondandolo18_2} yields the desired identity.
\end{proof}

 With the expression for the contact volume of $\tilde A$ at hand, we can proceed to prove the integral formula for the contact volume of $T^1\Oo$ analogously as in \cite[Proof of Lemma~4.6]{abbondandolo18}.
 
 \begin{proof}[Proof of Lemma~\ref{lemma:contactvolume}]
  With the help of Lemmas~\ref{lemma:flowsaturationvolume} and \ref{lemma:gen_fct_return_mapandtime}, we calculate
  \begin{align*}
   \vol(\tilde A) \ = \int_A \tau\omega &= \int_A (F(\eta)-\eta F'(\eta))\,  d\xi \wedge d\eta \\
   &= L \int_{-1}^1 (F(\eta) - \eta F'(\eta))\, d\eta \\
   &= 2L \int_{-1}^1 F(\eta)\, d\eta - L\eta F(\eta)\big|_{\eta=-1}^{\eta=1} \\
   &= 4L \int_0^1 F(\eta)\, d\eta - (m+n)L^2 - 2L \int_\Gamma \cos\beta\, d\beta \wedge ds \, ,
\intertext{where we also used the fact that $F$ is an even function. For the contact volume of the complement of $\tilde A$, by Lemma~\ref{lemma:flowsaturationset}, we obtain}
  \vol(T^1\Oo \setminus \tilde A)\ &= \vol( \{ u\in T^1\Oo \,|\, |K(u)| \ge r(s_0) \} ) \\[0.5ex]
  &= 2 \vol( \{ u \in T^1\Oo \,|\, K(u) \ge r(s_0) \} ) \\[0.5ex]
  &= 2 \vol( \R/2\pi\Z \times \Gamma ) \\[0.5ex]
  &= 2 \int\limits_{\R/2\pi\Z \times \Gamma} r(s)\, d\theta \wedge d\beta \wedge ds = 4\pi \int_\Gamma r(s)\, d\beta \wedge ds \, ,
  \end{align*}
   where we used expression \eqref{eq:hilbertcontactform} for the Hilbert contact form. The claim follows by adding the two identities above.
 \end{proof}

We conclude this section by proving the following statement which characterizes rotationally symmetric Besse spindle orbifolds in terms of their generating function.

\begin{prop}\label{prop:besse_iff_F_const}
 The generating function is constant if and only if $\Oo$ is Besse.
\end{prop}

\begin{proof}
 Again, we adapt the proof of the corresponding statement for smooth surfaces \cite[Lemma~4.7]{abbondandolo18} to the orbifold case.

Assume that $\Oo$ is Besse. Then, for a point $(\eta,\xi)\in A$ and for $(m+n)$ even, the point of first recurrence is $(\eta,\xi)$ itself, whereas for $(m+n)$ odd, it is the antipode $(\xi+\frac{L}{2},\xi)$. By Lemma~\ref{lemma:gen_fct_return_mapandtime}, this implies (independently of the parity of $(m+n)$) that $F'(\eta)\in L\Z$ for all $\eta\in\R/L\Z$. Consequently, as $F'$ is smooth, it is forced to be a constant integer multiple of $L$. Moreover, as $F'$ is odd, it has to be constantly zero. This proves that $F$ has to be constant. 

Now, assume that $F$ is constant and therefore have $F'(\eta)=0$. By Corollary~\ref{rmk:critpointsgeneratingfct}, for $\eta\in(-1,1)$ the corresponding geodesics are closed and thus in particular periodic. Consequently, so are the geodesics corresponding to $\eta=\pm 1$, that is, the equators. We note that so far we have shown all geodesics with starting vector in $A$ to be closed.

 One can now show that the fact that all geodesics starting on the equator $P_{s_0}$ are periodic implies that $r$ attains a local maximum at $s_0$. This is done by linearizing the geodesic equations \eqref{eq:geodesiceqations} at $P_{s_0}$ (we refer to our reference \cite[Proof of Lemma~4.7]{abbondandolo18} for details). Because $P_{s_0}$ was chosen as an equator of minimal length, we conclude that it is the unique equator of $\Oo$. By Lemma~\ref{lemma:flowsaturationset}, the only orbits of the geodesic flow that do not meet $A$ are those which parametrize the equator $P_{s_0}$. As those are closed curves as well, together with the preceding paragraph, we have shown $\Oo$ to be Besse.
\end{proof}

\section{Proof of the systolic inequalities}\label{sec:systolicineq}

With our preliminary considerations from the preceding sections, we are now able to prove our main result.

\begin{proof}[Proof of Theorem A]
 By Proposition~\ref{prop:systolic_ratio_besse} we only need to show that if $\Oo$ is not Besse, then its systolic ratio is less than $2(m+n)\pi$. That is, we have to show that there is a closed geodesic $\gamma$ whose lift to the unit tangent bundle is contractible such that
 \begin{align*}
  L(\gamma)^2 < \mathrm{area}(\Oo) \rho_\text{contr}(\Oo_\text{Besse}) = \frac{1}{2\pi} \vol(T^1\Oo) \rho_\text{contr}(\Oo_\text{Besse}) = (m+n)\vol(T^1\Oo) \, ,
 \end{align*}
 where we have used the relation \eqref{eq:areacontactvolume} between the Riemannian area of $\Oo$ and the contact volume of $T^1\Oo$. As in Section~\ref{sec:generatingfunction_prop} we let $A$ be the Birkhoff annulus associated to an equator $P_{s_0}$ of minimal radius $r(s_0)$ among all equators. By Corollary~\ref{rmk:critpointsgeneratingfct_contractiblegeodesic} to show the existence of a geodesic as described above it is sufficient to show the existence of a critical point $\eta_0$ of $F$ with critical value $\mu:=F(\eta_0)$ satisfying
 \begin{align*}
  \mu^2 < \frac{m+n}{4} \vol(T^1\Oo) \, .
 \end{align*}
 From Lemma~\ref{lemma:contactvolume} we know that
 \begin{align*}
  \vol(T^1\Oo) \ge 4L \int_0^1 F(\eta)\, d\eta - (m+n)L^2 + \int_\Gamma (4\pi r(s) - 2L \cos\beta)\, d\beta \wedge ds \, ,
 \end{align*}
 where
 \begin{align*}
  \Gamma := \{ (\beta,s) \in {\textstyle \left( - \frac \pi 2, \frac \pi 2 \right) } \times (0,M) \,|\, K(\beta,s) \ge r(s_0) \} \, .
 \end{align*}
 The Clairaut integral is given by $K(\beta,s) = r(s)\cos(s)$ and so we have $r(s) \ge r(s_0)$ on $\Gamma$. Consequently, we obtain
 \begin{align*}
  4\pi r(s) - 2L \cos\beta \ge 4\pi r(s_0) - 2L\cos\beta = 2L - 2L\cos\beta \ge 0 \quad \text{on } \Gamma \, ,
 \end{align*}
 and thus the second integral in the expression for the contact volume of $T^1\Oo$ is bounded from below by zero. This gives us the inequality
 \begin{align}\label{eq:contactvolume_ineq}
  \vol(T^1\Oo) \ge 4L \int_0^1 F(\eta)\, d\eta - (m+n)L^2 \, .
 \end{align}
 By Proposition~\ref{prop:iterates_contractible} we know that the $(m+n)$-fold iterate of the equator lifts to a contractible curve in $T^1\Oo$. Thus, we can assume the inequality
 \begin{align}\label{eq:equatorlengthbound}
  (m+n)^2 L^2 \ge (m+n)\vol(T^1\Oo)
 \end{align}
 to hold, for otherwise the equator would already be the curve we are trying to find. Combining inequalities \eqref{eq:contactvolume_ineq} and \eqref{eq:equatorlengthbound} we have arrived at the following bound for the integral over the generating function:
 \begin{align}\label{eq:generatingfunctionintegral}
  \int_0^1 F(\eta)\, d\eta \le \frac{m+n}{2} L \, .
 \end{align}
Since is $\Oo$ is not Besse, we recall from Proposition~\ref{prop:besse_iff_F_const} that $F$ is not constant. As a continuous function, $F$ attains a minimum on $[-1,1]$. Furthermore, we know that $F$ is positive and from Lemma~\ref{lemma:generating_function_bound} we get
 \begin{align*}
  F(1) \ge \frac{m+n}{2}L \, .
 \end{align*}
 This together with \eqref{eq:generatingfunctionintegral} implies that the minimum of $F$ on $[0,1]$ is attained for some $\eta_0\in[0,1)$ and takes a value $\mu=F(\eta_0)$ in the interval $(0,\frac{m+n}{2}L)$. Because $F$ is an even function, this minimum has to be a minimum of $F$ on $[-1,1]$ as well and therefore a critical value of $F|_{(-1,1)}$. From Lemma~\ref{lemma:generating_function_bound} and the fact that $\mu$ is the minimal value of $F$, we conclude
 \begin{align*}
  F(\eta) \ge \max \left\{\mu, \frac{m+n}{2}L\eta \right\} \quad \forall \eta\in[0,1] \, .
 \end{align*}
 Moreover, the above inequality has to be strict at $\eta_0 = \frac{2}{m+n} \frac{\mu}{L}$ as $F$ is a differentiable function. It follows that
\begin{align*}
 \int_0^1 F(\eta)\, d\eta &> \int_0^1 \max \left\{\mu, \frac{m+n}{2}L\eta \right\}\, d\eta \\
 &= \mu + \frac 12 \left( 1 - \frac{2}{m+n}\frac{\mu}{L} \right) \left( \frac{m+n}{2}L - \mu \right) = \frac{m+n}{4}L + \frac{1}{m+n}\frac{\mu^2}{L} \, .
\end{align*}
 Using this inequality together with \eqref{eq:contactvolume_ineq}, we get
\begin{align}\label{eq:contactvolumebound}
 \vol(T^1\Oo) \ge 4L \int_0^1 F(\eta)\, d\eta - (m+n)L^2 > \frac{4}{m+n}\mu^2 \, ,
\end{align}
 which is what we wanted to show.
\end{proof}

The proof of Theorem B has the same structure with some minor modification. We provide a condensed version for the convenience of the reader.

\begin{proof}[Proof of Theorem B]
 We first prove that $\rho_{\text{contr},2}$ for a rotationally symmetric Besse spindle orbifold is given as stated: From Corollary~\ref{thm:tannery_orbifold}, we know that the length of the equator is $2\pi$ and it has winding number $1$. Moreover, every other geodesic is closed with length $(m+n)\pi$ and has winding number $\frac{m+n}{2}$. Then according to Remark~\ref{rmk:windingnumber_homotopyclass}, the lift of any closed geodesic other than the equator to $T^1\Oo$ represents an element in the subgroup of the fundamental group of order $2$; the equator has to be iterated $\frac{m+n}{2}$ times for its lift to do so as well. Consequently, we have $\ell_{\text{min},k} = (m+n)\pi$. Together with the area of a Besse orbifold \eqref{eq:besseorbifold_area}, the expression for $\rho_{\text{contr},2}$ follows.

 Conversely, we have to show that if $\Oo$ is not Besse, then there is a closed geodesic $\gamma$ whose lift to $T^1\Oo$ represents an element in the subgroup of the fundamental group of order $2$ such that
 \begin{align*}
  L(\gamma)^2 &< \mathrm{area}(\Oo) \rho_{\text{contr},2}(\Oo_\text{Besse}) = \frac{m+n}{4}\vol(T^1\Oo) \, .
 \end{align*}
 Again, let $A$ be the Birkhoff annulus associated to an equator $P_{s_0}$ of minimal radius $r(s_0)$ among all equators. By virtue of Corollary~\ref{rmk:critpointsgeneratingfct_contractiblegeodesic} it suffices to find a critical point $\eta_0$ of $F$ such that
 \begin{align*}
  F(\eta_0)^2 < \frac{m+n}{4}\vol(T^1\Oo) \, .
 \end{align*}
 Inequality \eqref{eq:contactvolume_ineq} holds as before. If we again denote by $L$ the length of the equator, we can assume the inequality
 \begin{align*}
  \left(\frac{m+n}{2}L\right)^2 \ge \frac{m+n}{4}\vol(T^1\Oo)
 \end{align*}
 to hold, for otherwise the equator would already be the desired curve. Combining those two inequalities, again we obtain \eqref{eq:generatingfunctionintegral}. From here on we can conclude in the same way as in the proof of Theorem A.
\end{proof}

We are left to prove Theorem C. Before proving the bound, we readily compute the systolic ratio $\rho_{\frac{m+n}{2-\alpha}}(\Oo)$ for a Besse spindle orbifold $\Oo$ of type $S^2(m,n)$.
By Proposition~\ref{thm:tannery_orbifold} there are $\frac{m+n}{2-\alpha}-1$ closed geodesics shorter than the minimal common period $\frac{2(m+n)\pi}{2-\alpha}$ of all geodesics, namely the iterates of the equator of length $2\pi$. Consequently, we have $\tau_{\frac{m+n}{2-\alpha}} = \frac{2(m+n)\pi}{2-\alpha}$. Together with $\mathrm{area}(\Oo)=2\pi(m+n)$, see \eqref{eq:besseorbifold_area}, we obtain
\begin{align*}
 \rho_{\frac{m+n}{\alpha-2}}(\Oo) = \frac{2(m+n)\pi}{(2-\alpha)^2} \, .
\end{align*}
The bound is again obtained by the same arguments as in the proof of Theorem A:

\begin{proof}[Proof of Theorem C]
 We are left to prove that if $\Oo$ is not Besse, then we can find $\frac{m+n}{2-\alpha}$ sufficiently short closed geodesics, in the sense that their squared length satisfies
\begin{align*}
 \text{length}^2 < \mathrm{area}(\Oo) \rho_{\frac{m+n}{2-\alpha}}(\Oo) = \frac{\vol(T^1\Oo)}{2\pi} \rho_{\frac{m+n}{2-\alpha}}(\Oo) = \frac{m+n}{(2-\alpha)^2} \vol(T^1\Oo) \, .
\end{align*}  
Note that $\frac{1}{2-\alpha} = \frac{1+\alpha}{2}$. Hence, by Corollary~\ref{rmk:critpointsgeneratingfct_contractiblegeodesic} and the rotational symmetry it is sufficient to find a critical point $\eta_0$ of $F$ such that
\begin{align*}
 F(\eta_0)^2 < \frac{m+n}{4} \vol(T^1\Oo) \, .
\end{align*}
At the same time, we can assume that the first $\frac{m+n}{2-\alpha}$ iterates of the equator do not satisfy the desired length bound for otherwise we were done. That is, we can assume
\begin{align*}
 \frac{(m+n)^2}{(2-\alpha)^2} L^2 \ge \frac{m+n}{(2-\alpha)^2} \vol(T^1\Oo) \quad\Leftrightarrow\quad L^2 \ge \frac{\vol(T^1\Oo)}{m+n} \, ,
\end{align*}
where again $L$ denotes the length of the equator. We realize that the inequalities obtained so far coincide with those in the proof of Theorem A. As such, the rest of the argument carries over.
\end{proof}

\bibliographystyle{alpha}
\addtocontents{toc}{\protect\partbegin}
\bibliography{biblio}

\end{document}